\documentclass[12pt]{amsart}

\pdfoutput=1

\usepackage[text={400pt,660pt},centering]{geometry}

\usepackage{esint,amssymb}
\usepackage{graphicx}
\usepackage{mathtools}
\usepackage[colorlinks=true, pdfstartview=FitV, linkcolor=blue, citecolor=blue, urlcolor=blue,pagebackref=false]{hyperref}
\usepackage{microtype}

\usepackage{bm}
\usepackage{dsfont}

\makeatletter

\def\chaptermark#1{}

\def\chapter{%
  \vspace{1cm}
  \thispagestyle{plain}\global\@topnum\z@
  \@afterindenttrue \secdef\@chapter\@schapter}

\def\@chapter[#1]#2{\refstepcounter{chapter}%
  \ifnum\c@secnumdepth<\z@ \let\@secnumber\@empty
  \else \let\@secnumber\thechapter \fi
  \typeout{\chaptername\space\@secnumber}%
  \def\@toclevel{0}%
  \ifx\chaptername\appendixname \@tocwriteb\tocappendix{chapter}{#2}%
  \else \@tocwriteb\tocchapter{chapter}{#2}\fi
  \chaptermark{#1}%
  \addtocontents{lof}{\protect\addvspace{10\p@}}%
  \addtocontents{lot}{\protect\addvspace{10\p@}}%
  \@makechapterhead{#2}\@afterheading}
\def\@schapter#1{\typeout{#1}%
  \let\@secnumber\@empty
  \def\@toclevel{0}%
  \ifx\chaptername\appendixname \@tocwriteb\tocappendix{chapter}{#1}%
  \else \@tocwriteb\tocchapter{chapter}{#1}\fi
  \chaptermark{#1}%
  \addtocontents{lof}{\protect\addvspace{10\p@}}%
  \addtocontents{lot}{\protect\addvspace{10\p@}}%
  \@makeschapterhead{#1}\@afterheading}
\newcommand\chaptername{Part}

\def\@makechapterhead#1{\global\topskip 7.5pc\relax
  \begingroup
  \fontsize{\@xivpt}{18}\bfseries\centering
    \ifnum\c@secnumdepth>\m@ne
      \leavevmode \hskip-\leftskip
      \rlap{\vbox to\z@{\vss
          \centerline{\normalsize\mdseries
              \uppercase\@xp{
              }
              }
          \vskip 3pc}}\hskip\leftskip\fi
     \fontsize{\@xivpt}{14}\chaptername \enspace \thechapter. \enspace #1\par \endgroup
  \skip@34\p@ \advance\skip@-\normalbaselineskip
  \vskip\skip@ }
\def\@makeschapterhead#1{\global\topskip 7.5pc\relax
  \begingroup
  \fontsize{\@xivpt}{18}\bfseries\centering
  #1\par \endgroup
  \skip@34\p@ \advance\skip@-\normalbaselineskip
  \vskip\skip@ }
\def\appendix{\par
  \c@chapter\z@ \c@section\z@
  \let\chaptername\appendixname
  \def\thechapter{\@Alph\c@chapter}}

\newcounter{chapter}

\newif\if@openright

\makeatother
\renewcommand{\thechapter}{\Roman{chapter}} 

\newtheorem{proposition}{Proposition}
\newtheorem{theorem}[proposition]{Theorem}
\newtheorem{lemma}[proposition]{Lemma}

\theoremstyle{remark}
\newtheorem{remark}[proposition]{Remark}

\theoremstyle{definition}

\numberwithin{equation}{section}
\numberwithin{proposition}{section}
\numberwithin{figure}{section}
\numberwithin{table}{section}

\renewcommand{\le}{\leqslant}
\renewcommand{\ge}{\geqslant}

\newcommand{\Z}{\mathbb{Z}}
\newcommand{\N}{\mathbb{N}}
\newcommand{\R}{\mathbb{R}}

\newcommand{\E}{\mathbb{E}}
\renewcommand{\P}{\mathbb{P}}

\newcommand{\Zd}{{\mathbb{Z}^d}}
\newcommand{\Rd}{{\mathbb{R}^d}}
\newcommand{\B}{\mathbb{B}}

\newcommand{\PP}{\mathbf{P}}
\newcommand{\EE}{\mathbf{E}}

\newcommand{\eps}{\varepsilon}

\renewcommand{\a}{\mathbf{a}}

\newcommand{\ahom}{{\overbracket[1pt][-1pt]{\a}}}  
\newcommand{\had}{{\hat \a_\de}}

\renewcommand{\subset}{\subseteq}

\newcommand{\Ll}{\left}
\newcommand{\Rr}{\right}
\renewcommand{\d}{\mathrm{d}}

\DeclareMathOperator{\dist}{dist}

\DeclareMathOperator{\diag}{diag}

\newcommand{\var}{\mathbb{V}\mathrm{ar}}

\newcommand{\X}{\mathcal{X}}

\renewcommand{\tilde}{\widetilde}
\renewcommand{\hat}{\widehat}

\newcommand{\mcl}{\mathcal}

\newcommand{\al}{\alpha}
\newcommand{\be}{\beta}

\newcommand{\de}{\delta}

\newcommand{\1}{\mathds{1}}

\newcommand{\e}{\mathbf{e}}
\newcommand{\Lpot}{L^2_{\mathrm{pot}}}
\renewcommand{\L}{{\mathcal L}}
\newcommand{\T}{\mathsf T}

\newcommand{\la}{\left\langle}
\newcommand{\ra}{\right\rangle}

\begin{document}

\title[Efficient methods for estimation of homogenized coefficients]{Efficient methods for the estimation of homogenized coefficients}

\begin{abstract}
The main goal of this paper is to define and study new methods for the computation of effective coefficients in the homogenization of divergence-form operators with random coefficients. The methods introduced here are proved to have optimal computational complexity, and are shown numerically to display small constant prefactors. In the spirit of multiscale methods, the main idea is to rely on a progressive coarsening of the problem, which we implement via a generalization of the Green-Kubo formula. The technique can be applied more generally to compute the effective diffusivity of any additive functional of a Markov process. In this broader context, we also discuss the alternative possibility of using Monte-Carlo sampling, and show how a simple one-step extrapolation can considerably improve the performance of this alternative method.
\end{abstract}

\author[J.-C. Mourrat]{Jean-Christophe Mourrat}
\address[J.-C. Mourrat]{Ecole normale sup\'erieure de Lyon, CNRS, Lyon, France}
\email{jean-christophe.mourrat@ens-lyon.fr}

\date{\today}

\maketitle

%
%
%
%
%
%

\section{Introduction}

Let $-\nabla \cdot \a(x) \nabla$ be a divergence-form operator with random coefficients. If the law of the matrix field $(\a(x))$ is stationary and ergodic, then the operator homogenizes over large scales to an operator with constant, deterministic coefficients $-\nabla \cdot \ahom \nabla$. The main goal of this paper is to explore numerical methods to estimate the homogenized matrix $\ahom$. We can formalize this task as follows.

\smallskip

\noindent \textbf{Problem.} \emph{Find an algorithm that, given $\de > 0$ and a realization of the random coefficient field $(\a(x))$, outputs a matrix $\hat \a_\de$ such that}
\begin{equation}  
\label{e.prob.algo}
\E \Ll[|\ahom - \hat \a_\de|^2  \Rr]^\frac 1 2 \le \de.
\end{equation}

\smallskip 

In \eqref{e.prob.algo}, we chose to measure the quality of the approximation $\hat \a_\de$ in an $L^2$ sense for definiteness, but other choices (e.g.\ convergence in probability, or higher moments) are equally valid, and in fact would only marginally alter the results presented below. The point is that $\had$ should typically be within $\de$ or less from $\ahom$. 

\smallskip

Naturally, the real question is to find an \emph{efficient} algorithm. We focus on a discrete-space setting, and assume throughout that the coefficients are uniformly elliptic, independent and identically distributed (i.i.d.). In order to measure computational effort, we use a slightly loosely defined notion of ``elementary'' operation: any memory access, boolean operation, floating point addition or multiplication counts as one operation. We first give a lower bound on the number of operations any algorithm must use.
\begin{proposition}[Complexity lower bound]
\label{p.lower.bound}
No algorithm can output an approximation $\hat \a_\de$ of $\ahom$ in the sense of \eqref{e.prob.algo} within $o(\de^{-2})$ operations.
\end{proposition}
The main method we propose to compute an approximation of $\ahom$ is as follows. We fix $\xi \in \Rd$ of unit norm, set
\begin{equation*}  
v_{-1}(x) := (\nabla \cdot \a\xi)(x) \qquad (x \in \Z^d),
\end{equation*}
and for each $k \in \N$, we let $v_k$ be the solution to
\begin{equation}  
\label{e.def.vk}
(2^{-k} - \nabla \cdot \a \nabla)v_k = 2^{-k} v_{k-1} \qquad \text{on } \Zd. 
\end{equation}
(We refer to Section~\ref{s.defs} for precise notation and assumptions.) Note that the functions $(v_k)_{k \ge -1}$ depend on the coefficient field $\a$, although we keep this dependence implicit. For each $r \ge 0$, we write 
\begin{equation*}  
B_r := \{-\lfloor r \rfloor, \ldots, \lfloor r \rfloor\}^d.
\end{equation*}
\begin{theorem}[Efficient approximation of $\ahom$]
\label{t.main}
For each $\eps \in (0, \frac{d-1}{2d})$, there exists a constant $C < \infty$ such that the following holds. Fix $n \in \N$, and denote
\begin{equation}  
\label{e.def.main.rk}
r_k := 2^{n- \Ll( \frac 1 2 - \eps \Rr) k} \qquad (k \in \{0,\ldots,n\}),
\end{equation}
\begin{equation}  
\label{e.def.final.hatsigma}
\hat \sigma^2_n := \sum_{k = 0}^n \frac{1}{|B_{r_k}|} \sum_{x \in B_{r_k}} 2^k \Ll( v_{k-1}(x) v_k(x) + v_k^2(x) \Rr).
\end{equation}
We have 
\begin{equation}  
\label{e.main}
\E \Ll[ \Ll( \xi \cdot \ahom \xi - \E[\xi \cdot \a \xi] + \hat \sigma^2_n \Rr) ^2	 \Rr] ^\frac 1 2 \le C 2^{-\frac {dn}{2}}.
\end{equation}
Moreover, the quantity $\hat \sigma^2_n$ can be computed in at most $C n 2^{dn}$ operations. 
\end{theorem}
By the central limit theorem, the quantity $\E[\xi \cdot \a \xi]$ can be computed at precision $\delta > 0$ in $C \de^{-2}$ operations. Hence, Theorem~\ref{t.main} gives us a method that requires $\de^{-2} \log(\de^{-1})$ operations to compute an approximation of $\ahom$ at precision $\de > 0$ in the sense of~\eqref{e.prob.algo}, essentially matching the lower bound given by Proposition~\ref{p.lower.bound}.

\smallskip

For $\eps = 0$, the statement of Theorem~\ref{t.main} remains valid provided that we replace the right side of \eqref{e.main} by $C n 2^{-\frac{dn}{2}}$. I expect that the sharp upper bound in this case is $C n^\frac 1 2 2^{-\frac{dn}{2}}$. Whatever may be, this additional multiplicative factor of $n$ or $n^\frac 1 2$ does not degrade the bound much. For simplicity, it is the version with $\eps = 0$ that was implemented numerically. 

\smallskip

The proof of Theorem~\ref{t.main} relies in particular on the assumption that the law of the coefficient field $\P$ is invariant under the action of $\Z^d$. If this assumption is weakened and the law is only assumed to be invariant under the action of a sublattice of $\Z^d$ which may be unknown to us a priori, then one can use masks in the spatial averages over $B_{r_k}$ appearing in \eqref{e.def.final.hatsigma} in order to avoid boundary layer effects. When weakening the independence assumption on the coefficients, I expect Theorem~\ref{t.main} to remain unchanged provided that the field $x \mapsto \a(x)$ is sufficiently uncorrelated for partial sums over boxes to satisfy a central limit theorem.

\smallskip

The power of the method presented in Theorem~\ref{t.main} comes from the fact that the domain of interest shrinks rapidly as $k$ increases. Moreover, when $k$ is of order one, the condition number of the elliptic problems we need to solve and the size of the associated boundary layers are also of order one. As $k$ increases, condition numbers and boundary layers become larger, but these adverse effects are more than compensated by the reduction of size of the domain of interest. In fact, in order to prove the complexity upper bound in Theorem~\ref{t.main}, we do not need to use any preconditioner for the resolution of the elliptic problems: a direct application of the conjugate gradient method suffices. Moreover, the overhead caused by the presence of boundary layers only participates to a negligible fraction of the total computational time: this fraction of time is of the order of~$2^{-\frac n 2}$ in dimension $d = 2$, and of the order of~$2^{- n}$ in dimension $d \ge 3$.

\smallskip

It is straightforward to adapt the method of Theorem~\ref{t.main} to accommodate for parallel computing. Indeed, we will see that the correlation length of the random field
\begin{equation}  
\label{e.mixing.field}
x \mapsto v_{k-1}(x) v_k(x) + v_k^2(x)
\end{equation}
is of the order of $2^{\frac k 2}$. Computing spatial averages of this random field is the best strategy if all computations are done sequentially. On the other hand, if parallel computing is available, then we may proceed differently and compute separately $v_k^{(1)}, v_k^{(2)}, \ldots$, based on independent copies of the coefficient field but otherwise defined in the same way as $v_k$, and then average the results. The general rule is that we may replace any spatial average
\begin{equation*}  
\frac 1 {|B_r|} \sum_{x \in B_r} \Ll( v_{k-1}(x) v_k(x) + v_k^2(x)  \Rr) 
\end{equation*}
by
\begin{equation*}  
\frac{1}{\lceil 2^{-\frac {kd} 2} r^d \rceil} \sum_{i = 1}^{\lceil 2^{-\frac {kd} 2} r^d \rceil} \Ll( v_{k-1}^{(i)}(0) v_k^{(i)}(0) + \Ll(v_k^{(i)}(0)\Rr)^2  \Rr) ,
\end{equation*}
or with any other intermediate combination of spatial averaging and independent sampling. 

\smallskip

Disregarding logarithmic factors, the previously known best procedure for computing $\ahom$ at precision $\de > 0$ consists in averaging $\de^{-1}$ samples of the energy of an approximate corrector calculated on a box of side length $\de^{-\frac 1 d}$---see \cite{Gloria} and \cite[Section~3.3.4]{cemracs} for details. If no preconditioner is used, then the computational cost of this method is of the order of $\de^{-\Ll(2 + \frac 1 d\Rr)}$, up to logarithmic factors. While efficient preconditioners can bring this estimate back to being of the order of~$\de^{-2}$, up to logarithmic factors, the prefactors involved may be very large in practice. I did some modest tests with an incomplete Cholesky factorization as a preconditioner. While this did reduce the number of iterations involved in the conjugate gradient method, it was not sufficient to compensate for the induced overhead and resulted in degraded performance. Other preconditioners will certainly show better performance (using for instance the homogeneous Laplacian, which can be solved in almost linear time by fast Fourier transform), and I did not investigate this question further. In any case, the computational overhead caused by boundary layers for this older method in fact takes up an asymptotically full proportion of the computational time, since boundary layers occupy a volume of the order of $\de^{-1} \log^d(\de^{-1})$ on each of the $\de^{-1}$ domains. Naturally, this observation calls for an adjustment of the strategy: one should solve fewer problems, each of larger volume. But this runs counter to the fact that it is computationally best to solve smaller elliptic problems, and only suboptimal compromises can be found in this way.

\smallskip 

An additional advantage of the method described in Theorem~\ref{t.main} is that it is cumulative: all computations involved in the evaluation of $\hat{\sigma}^2_n$ remain useful if the approximation needs to be refined. In contrast, in the older method, computations done for a coarse precision must be thrown away if greater precision is then desired.

\smallskip

The theoretical arguments presented here are supplemented by numerical tests. The numerical results confirm the theoretical predictions, and display small constant prefactors. 

\smallskip

As was said above, we assume that the coefficient field $(\a(x))$ is uniformly elliptic and i.i.d., and we describe algorithms that perform well under these assumptions. In practice, it seems more appropriate to implement more ``agnostic'' algorithms, that would evaluate on the fly the behavior of the mean and variance of the spatial averages appearing in \eqref{e.def.final.hatsigma} as $r$ and $k$ vary, and adapt the numerical scheme dynamically. We will not discuss this possibility any further here. 

\smallskip

It would be interesting to investigate how the new method presented here adapts to the continuous-space setting. In this case, I expect that the error analysis remains unchanged. The complexity analysis would however also have to take into account the cost of the small-scale resolution of the equations, below the typical length scale of the correlations of the coefficient field. This problem is separate from the homogenization phenomenon, and is likely to come as the same multiplicative factor in the complexity analysis of any method.

\smallskip

Another interesting line of further research concerns the behavior of the effective conductivity of percolation clusters. We refer to \cite{hughes,kum-stflour} for very nice surveys on this problem. The method presented here should allow for a sharper numerical analysis of the effective conductivity as the percolation probability approaches criticality. Despite immense progress on the understanding of critical and near-critical percolation in two dimensions, it is still not known whether the effective conductivity behaves as a power law near criticality (let alone compute the exponent) in small dimensions. See however \cite{kesten-subdiff,jarnac,damronetal}, as well as \cite{barandco,koznac, BACF1, BACF2} for very fine results in high dimensions.

\smallskip

It turns out that the underlying idea of the method can be recast in the more general framework of additive functionals of Markov processes. The paper is thus split into two parts. In the first part, we describe the general strategy in this broader context. We also discuss alternative methods based on Monte-Carlo sampling, and show that a simple one-step extrapolation can in general improve the efficiency of this approach dramatically. The second part of the paper is concerned with the computation of $\ahom$ per se. In particular, we prove Theorem~\ref{t.main} and report on numerical results there. Readers who are only interested in the problem of computing $\ahom$ are encouraged to go directly to the second part of this paper, where we explain how to extract relevant information from the first part without having to refer to Markov processes. 

\smallskip

We conclude this introduction with a brief review of related work. The core idea of the method presented in Theorem~\ref{t.main} is to decompose the problem into a series of scales. 
The approach is inspired by the renormalization scheme introduced in \cite{AKM2}, which progressively coarsens the operator and relies on a ``quadratic response'' behavior around the homogenization limit---see also the discussion below \eqref{e.basis}. Other relevant recent references for quantitative homogenization include \cite{vardecay,GO1,GO2,GNO,MO,AS,GNO3,AM,AKM,GO5,AKMBook}.

\smallskip

More generally, the approach presented here is similar in spirit to multiscale methods \cite{brandt, eq-free, hmm}. In this context, it was  quickly realized that the resolution of cell problems or the computation of adapted finite-element base functions can display so-called ``resonant errors'' due to inappropriate boundary condition \cite{hw, hwc,yue-e,eh-book}. A powerful approach was introduced in \cite{blb,GO2,Gloria, GloriaH, vardecay,approx}, based on the introduction of a massive term that smoothly makes the faraway boundary condition irrelevant. The methods presented here are faithful to this approach; see also \cite{malpet, henpet, time-dep}. 

\smallskip

A different line of research focused on developing more efficient methods for computing homogenized coefficients of small random perturbations of a simpler, e.g.\ homogeneous, medium, in the spirit of earlier fundamental insights \cite{maxwell, rayleigh}. A typical example is that of spherical inclusions randomly placed at a small volume fraction $p >0$ in an otherwise constant medium. One looks for asymptotic formulas of the homogenized coefficients involving easily computable coefficients, in the regime $p \ll 1$. We refer to \cite{kozlov, papa-survey, bm1, alb1,alb2,almog1,almog2, dl-diff, DG} for a mathematical analysis of the problem. Note that the seemingly simpler task of determining the value of the volume fraction $p$ at precision $\de$ from a sample of the coefficient field already requires of the order of $\de^{-2}$ operations (the proof of the lower bound is a minor variant of the proof of Proposition~\ref{p.lower.bound}).

\smallskip

Lastly, several techniques have been explored to reduce the size of the fluctuations of quantities of interest, typically by a constant multiplicative factor. We refer to \cite{lebris-survey} for a recent survey.

\chapter{General methods}
\label{part.one}

\section{Main results}

In this first part of the paper, we explore the problem of computing the effective diffusivity of general additive functionals of a Markov process. 
Let $(X(t))_{t \ge 0}$ be a Feller process taking values in a metric space $\mcl X$, and denote the associated family of probability laws by $(\PP_x)_{x \in \mcl X}$, with corresponding expectations $(\EE_x)_{x \in \mcl X}$. We assume that there exists a probability measure $\P$ on~$\mcl X$ (with corresponding expectation $\E$) such that the process $(X_t)$ is reversible and ergodic with respect to $\P$; we denote the law of $(X(t))$ started from the measure~$\P$ by $\PP_\P$ (with corresponding expectation $\EE_\P$). The assumption of reversibility will be dropped shortly, and is only meant to facilitate the exposition here. Given $f \in L^2(\P)$ of mean zero, one can check using reversibility that the limit 
\begin{equation}  
\label{e.def.H-1}
\sigma^2(f) :=  \lim_{t \to + \infty}  \EE_\P \Ll[ \Ll( \frac 1 {\sqrt{t}}\int_0^t f(X(s)) \ \d s \Rr)^2  \Rr]  \in [0,+\infty]
\end{equation}
always exists. Denote by $\L$ the (non-negative) infinitesimal generator of $(X(t))$, and $\langle f, g \rangle := \E[fg]$. An equivalent characterization of $\sigma^2(f)$ is given by
\begin{equation}  
\label{e.Poisson}
\sigma^2(f) = 2 \la  f, \L^{-1} f \ra := 2 \lim_{\lambda \to 0} \la f, (\lambda + \L)^{-1} f \ra.
\end{equation}
Kipnis and Varadhan \cite{kipvar} showed that the condition $\sigma^2(f) < \infty$ implies that the additive functional
\begin{equation}  
\label{e.additive.func}
\frac 1 {\sqrt{t}} \int_{0}^t f(X(s)) \, \d s
\end{equation}
converges in law under $\PP_\P$ to a centered Gaussian random variable of variance~$\sigma^2(f)$ as $t$ tends to infinity. 

\smallskip

Our goal is to discuss practical ways to compute $\sigma^2(f)$. There are two standard methods to address this question (see for instance \cite{cemracs}). The first one uses the formula \eqref{e.def.H-1} as the basis for a Markov chain Monte-Carlo (MCMC) algorithm. The second method dispenses with the simulation of the Markov chain, and uses the formula \eqref{e.Poisson} instead. The problem then becomes that of computing an approximation of $\L^{-1} f$, and computing the average. 

\smallskip

We study MCMC methods in Section~\ref{s.mcmc}. The main contribution of this section is to show that a one-step extrapolation of the naive MCMC algorithm yields significant improvement of the approximation in most cases.

\smallskip

We then explore methods based on the formula \eqref{e.Poisson}. Similarly to multiscale methods, we aim to split the computational effort involved in the computation of~$\L^{-1} f$ into a hierarchy of scales. A good starting point for witnessing this multiscale decomposition is the  Green-Kubo formula, which in our present context is a direct consequence of \eqref{e.Poisson} and reads
\begin{equation}
\label{e.gk1}
\sigma^2(f) = 2 \int_0^{+\infty} \la f, e^{-t\L} f \ra \, \d t  = 2 \int_0^{+\infty} \EE_\P \Ll[ f(X(0))  f(X(t))  \Rr] \, \d t.
\end{equation}
We think of the part of the integral with $t \lesssim 1$ as carrying the small-scale information of the process, while the part with $t \gg 1$ carries the larger-scale information.
The first, very elementary but crucial, modification to \eqref{e.gk1} we then bring forward is that, by reversibility,
\begin{equation}
\label{e.reversible}
\la f, e^{-t \L} f \ra = \la e^{-\frac t 2 \L}f, e^{-\frac{t}{2}\L} f \ra
\end{equation}
so that, after a change of variables,
\begin{align}
\label{e.gk2}
\sigma^2(f) 
&  = 4 \int_0^{+\infty} \la e^{-t \L}f, e^{-t\L} f \ra \, \d t .
\end{align}
The point of this simple modification is that in order to compute $\E[(e^{-t\L} f)^2]$, we need to average over samples of $(e^{-t\L} f)^2$, and for large $t > 0$, each of these samples will be much smaller than realizations of $f e^{-t\L} f$ (by a square factor). For large $t > 0$, we thus need much fewer samples to compute the average on the right side of \eqref{e.reversible} than we do to compute the average on the left side of \eqref{e.reversible}. This is relevant for practical computations since it is usually more difficult to compute samples of $e^{-t\L} f$ when $t$ is large. In other words, we decomposed the problem of computing $\sigma^2(f)$ into a small-scale part ($t \lesssim 1$) which has a contribution of order one and is an average over many samples that are relatively easy to compute, and a large-scale part ($t \gg 1$) which refines the estimate, is computationally more demanding, but requires much fewer samples due to the own smallness of each of these samples. This is what will ultimately allow for the rapid shrinking of the domain of interest in \eqref{e.def.final.hatsigma} as $k$ increases. 

\smallskip



The second modification to the Green-Kubo formula we consider is meant to replace the integral over time by a sum over a moderate number of terms. The starting point is the resolvent formula, which states that for every $\lambda, \mu > 0$,
\begin{equation}
\label{e.resolvent}
R_\lambda  = R_\mu + (\mu - \lambda) R_\lambda R_\mu, \quad \text{where } R_\lambda := (\lambda  + \L)^{-1}.
\end{equation}
For any given sequence $(\mu_k)_{k \in \N}$ of real numbers in $(0,1]$, we thus apply \eqref{e.resolvent} recursively and obtain, at least formally,
\begin{equation}  
\label{e.sum.resolv}
R_0 = \sum_{k = 0}^{+\infty} \mu_0 \cdots \mu_{k-1} R_{\mu_0} \cdots R_{\mu_k}.
\end{equation}
As will be explained in more details below (see Theorem~\ref{t.pgk}), defining recursively
\begin{equation*}  
f_{-1} := f \quad \text{and} \quad \forall k \in \N, \quad f_k := \mu_k R_{\mu_{k}} f_{k-1},
\end{equation*}
we can use the formula \eqref{e.sum.resolv} and reversibility to get
\begin{equation}  
\label{e.induc.Poisson}
\la f, \L^{-1} f \ra = \sum_{k = 0}^{+\infty} \mu_k^{-1} \Ll( \la f_{k-1}, f_k \ra + \la f_k, f_k \ra \Rr).
\end{equation}
This is the variant of the Green-Kubo formula that we will use. The choice of the parameters $\mu_k$ can be adjusted according to the problem at hand. Here we mostly have in mind dynamics that relax to equilibrium at a polynomial rate. In this case, it seems appropriate to choose $\mu_k$ to scale to $0$ geometrically, e.g.\ by fixing $\mu_k \equiv 2^{-k}$. While the series is infinite, we will show that the partial sums typically converge to their limit exponentially fast. 

\smallskip

In Section~\ref{s.1notation}, we drop the reversibility assumption and define our notation precisely. We discuss Monte-Carlo methods in Section~\ref{s.mcmc}. We then prove \eqref{e.induc.Poisson} and, under suitable assumptions, give estimates on the rate of convergence of the partial sums to their limit in Section~\ref{s.variants.gk}. We show how to adapt these considerations to discrete-time Markov chains in Section~\ref{s.dynsys}. Finally, in Section~\ref{s.examples}, we discuss briefly and at a heuristic level the comparative advantages of the different methods for interacting particle systems and Langevin dynamics in random potential.

%
%
%
%
%
%

\section{Notation and assumptions}
\label{s.1notation}

In this section, we drop the reversibility assumption and clarify some of the notation used in the previous section. 

\smallskip

From now on, we only assume that the measure $\P$ is invariant for the Markov process $X$. We still denote by $\L$ the infinitesimal generator of the process, and choose the sign convention so that $(e^{-t\L})_{t \ge 0}$ is the associated semigroup. That is, for every bounded and continuous function $f : \mcl X \to \R$ and every $t \ge 0$, we have
\begin{equation*}  
(e^{-t \L} f)(x) = \EE_x \Ll[ f(X_t) \Rr] .
\end{equation*}
By Jensen's inequality, for each $t \ge 0$ and $p \in [1,\infty]$, we have that $e^{-t\L}$ is a contraction from $L^p(\P)$ to itself. It thus follows that for each $\lambda > 0$ and $p \in [1,\infty]$, the resolvent 
\begin{equation}  
\label{e.def.R}
R_\lambda  := \int_0^{+\infty} e^{-\lambda t} e^{-t \L} \, \d t
\end{equation}
is a well-defined operator from $L^p(\P)$ to itself, and moreover, $\lambda R_\lambda$ is a contraction from $L^p(\P)$ to itself. One can then verify that $R_\lambda = (\lambda + \L)^{-1}$ and that the resolvent formula \eqref{e.resolvent} holds (see e.g.\ \cite[Theorem~3.16 and (3.7)]{lig1}). In particular, for every $\lambda > 0$ and $\mu > 0$, the resolvents $R_\lambda$ and $R_\mu$ commute. 
We denote by $\L^*$ the adjoint of $\L$ in $L^2(\P)$, and by $(e^{-t\L^*})_{t \ge 0}$ and $(R_\lambda^*)_{\lambda > 0}$ the associated semigroup and resolvent. 

\smallskip

We fix a mean-zero function $f \in L^2(\P)$. In the non-reversible setting, it is more difficult to assert whether the additive functional 
\begin{equation}
\label{e.add.func2}
\frac 1 {\sqrt{t}} \int_{0}^t f(X(s)) \, \d s
\end{equation}
satisfies a central limit theorem. We refer to \cite[Chapter~2]{klo} for several results in this direction. We only record here that under relatively weak conditions, the limit 
\begin{equation}  
\label{e.def.sigma2}
\sigma^2(f) := 2 \lim_{\lambda \to 0} \la f, (\lambda + \L)^{-1} f \ra
\end{equation}
is well-defined, and as $t$ tends to infinity, the random variable in \eqref{e.add.func2} converges in law under $\PP_\P$ to a centered Gaussian of variance $\sigma^2(f)$ (see in particular \cite[Theorem~2.7 and (2.14)]{klo}). 

\smallskip

Part of the difficulty with non-reversible dynamics is that quantities such as $\la f, (\lambda + \L)^{-1} f \ra$ or $\la f, e^{-t\L} f \ra$ are no longer necessarily nonnegative. Since the emphasis of the present paper is on numerical methods for the computation of~$\sigma^2(f)$ rather than on subtleties that may arise in bordeline cases, we make the simplifying assumption that
\begin{equation}  
\label{e.hyp.int}
\int_0^{+\infty} \big| \la f, e^{-t \L} f \ra \big| \, \d t < \infty.
\end{equation}
Under this assumption, the existence and finiteness of the limit in the definition of $\sigma^2(f)$, see \eqref{e.def.sigma2}, are easily shown using \eqref{e.def.R}. The validity of the classical Green-Kubo formula~\eqref{e.gk1} is obtained similarly.

\smallskip

Whenever we want to go beyond the derivation of identities and want to discuss rates of convergence, we will assume that there exists $\al > 1$ such that for every $t \ge 0$,
\begin{equation}  
\label{e.assump.decay}
\big|\la f, e^{-t\L} f \ra\big| \le (1+t)^{-\alpha}.
\end{equation}
Note that a possible multiplicative constant in this inequality could be absorbed after a redefinition of $f$. We will indicate explicitly whenever \eqref{e.hyp.int} or \eqref{e.assump.decay} is assumed. 


%
%
%
%
%
%

\section{Monte-Carlo methods}
\label{s.mcmc}

In this section, we show that a single-step extrapolation of the naive Monte-Carlo estimator for $\sigma^2(f)$ allows for a significant improvement in computational complexity. In addition to the quantitative assumption \eqref{e.assump.decay}, we also assume that
\begin{equation}
\label{e.fourth}
\sup_{t \ge 1} \EE_\P \Ll[ \Ll( \frac 1 {\sqrt{t}} \int_0^t f(X(s)) \, \d s \Rr)^4  \Rr] \le 1.
\end{equation}
The assumption of \eqref{e.hyp.int} implies that the supremum on the left side of \eqref{e.fourth}, with the fourth power replaced by the second power, is finite. The assumption of finiteness of the fourth moment in \eqref{e.fourth} (which can then normalized to be of unit size, up to a redefinition of $f$) will be used to control the fluctuating part of the Monte-Carlo estimator. 

\smallskip

In order to describe MCMC methods, we assume that we are given a family $X^{(0)}, X^{(1)}, \ldots$ of processes that are independent and are each distributed according to the law $\PP_\P$ (that is, the law of $X$ started from the measure $\P$). We denote their joint law by $\PP_\P^{\otimes}$ (expectation $\EE_\P^\otimes$). The next proposition quantifies the accuracy of the naive Monte-Carlo estimator. 

\begin{proposition}[naive MCMC method]
\label{p.naive}
For each $\al > 1$, there exists a constant $C < \infty$ such that the following holds. Let $f \in L^2(\P)$ satisfy \eqref{e.assump.decay} and~\eqref{e.fourth}, and for each $t \ge 2$ and positive integer $N$, let
\begin{equation}  
\label{e.def.hat0}
\hat{\sigma}_0^2(N,t,f) := \frac 1 N \sum_{k = 1}^{N} \Ll(\frac 1 {\sqrt{t}}\int_0^t f(X^{(k)}(s)) \, \d s\Rr)^2.
\end{equation}
We have
\begin{equation}  
\label{e.naive}
\EE_\P^{\otimes} \Ll[ \Ll( \hat{\sigma}_0^2(N,t,f) - \sigma^2(f) \Rr) ^2 \Rr] ^\frac 1 2 \le C 
\Ll|
\begin{array}{ll}
N^{-\frac 1 2} + t^{-(\al -1)} & \quad \text{if } \al < 2, \\
N^{-\frac 1 2} + t^{-1} \log(t) & \quad \text{if } \al = 2, \\
N^{-\frac 1 2} + t^{-1}  & \quad \text{if } \al > 2.
\end{array}
\Rr.
\end{equation}
\end{proposition}
\begin{proof}
We decompose the left side of \eqref{e.naive} into a variance and a bias terms:
\begin{multline}  
\label{e.varbias}
\EE_\P^{\otimes} \Ll[ \Ll( \hat{\sigma}_0^2(N,t,f) - \sigma^2(f) \Rr) ^2 \Rr] \\
= \EE_\P^{\otimes}  \Ll[ \Ll( \hat{\sigma}_0^2(N,t,f) - \EE_\P^{\otimes} \Ll[\hat{\sigma}_0^2(N,t,f) \Rr]\Rr)^2 \Rr] + |\EE_\P^{\otimes} \Ll[\hat{\sigma}_0^2(N,t,f) \Rr] - \sigma^2(f)|^2.
\end{multline}
By \eqref{e.fourth}, it is clear that the first term on the right side above is bounded by~$N^{-1}$. By stationarity, the expectation appearing in the second term can be rewritten as
\begin{align}  
\notag
\EE_\P\Ll[\Ll(\frac 1 {\sqrt t} \int_0^{t} f(X(s)) \, \d s\Rr)^2 \Rr] & = \frac{2}{t} \int_{0 \le u < v \le t} \EE_\P \Ll[ f(X(u)) f(X(v)) \Rr]  \, \d u \, \d v \\
\notag
& = \frac{2}{t}\int_{0 \le u < v \le t} \EE_\P \Ll[ f(X(0)) f(X(v-u)) \Rr]   \, \d u \, \d v  \\
\label{e.naivebias}
& = \frac{2}{t} \int_0^t (t-s) \la f, e^{-s \L} f  \ra \, \d s.
\end{align}
Recalling the formula \eqref{e.gk1} for $\sigma^2(f)$ and using \eqref{e.assump.decay}, we can bound the difference between the integral above and $\sigma^2(f)$ by
\begin{equation*}  
\frac 2 t \int_0^t s (1+s)^{-\al} \, \d s + 2 \int_t^{+\infty} (1+s)^{-\al} \, \d s \le 
C \Ll|
\begin{array}{ll}
t^{-(\al -1)} & \quad \text{if } \al < 2, \\
t^{-1} \log(t) & \quad \text{if } \al  = 2, \\
t^{-1} & \quad \text{if } \al > 2.
\end{array}
\Rr.
\end{equation*}
This completes the proof.
\end{proof}

The next proposition, which is the main result of this section, shows that a simple one-step extrapolation of the naive MCMC estimator yields an improved convergence rate in most cases.

\begin{proposition}[Improved MCMC method]
\label{p.improved}
For each $\al > 1$, there exists a constant $C < \infty$ such that the following holds. Let $f \in L^2(\P)$ satisfy \eqref{e.assump.decay} and~\eqref{e.fourth}, and for each $t \ge 2$ and positive integer $N$, let 
\begin{equation*}  
\hat{\sigma}_1^2(N,t,f) := 2\hat{\sigma}_0^2(N,2t,f) - \hat{\sigma}_0^2(N,t,f),
\end{equation*}
where $\hat{\sigma}_0^2(N,t,f)$ is defined in \eqref{e.def.hat0}. We have
\begin{equation*}  
\EE_\P^{\otimes} \Ll[ \Ll( \hat{\sigma}_1^2(N,t,f) - \sigma^2(f) \Rr) ^2 \Rr] ^\frac 1 2 \le C \Ll(N^{-\frac 1 2} + t^{-(\al - 1)}\Rr).
\end{equation*}
\end{proposition}
\begin{proof}
We use the variance-bias decomposition \eqref{e.varbias} (with $\hat{\sigma}_0^2$ replaced by $\hat{\sigma}_1^2$) and estimate the variance as in the proof of Proposition~\ref{p.naive}. There remains to estimate the bias term, which takes the form
\begin{equation*}  
\Ll|2\EE_\P\Ll[\Ll(\frac 1 {\sqrt {2 t}} \int_0^{2t} f(X(s)) \, \d s\Rr)^2 \Rr] - \EE_\P\Ll[\Ll(\frac 1 {\sqrt t} \int_0^{t} f(X(s)) \, \d s\Rr)^2 \Rr] - \sigma^2(f) \Rr|.
\end{equation*}
In view of \eqref{e.gk1} and \eqref{e.naivebias}, we can rewrite this quantity as
\begin{align*}  
2\Ll|  \int_0^{+\infty} \Ll[2\Ll( 1 - \frac s {2t} \Rr)_+ - \Ll( 1 - \frac s t \Rr)_+ - 1 \Rr] \la f, e^{-s\L} f \ra \, \d s \Rr|,
\end{align*}
where we denote $x_+ = x \vee 0 = \max(x,0)$ for the positive part of $x$. We also write $x \wedge y = \min(x,y)$. 
Simplifying and using~\eqref{e.assump.decay}, we get that the expression above is
\begin{align}  
\label{e.extrap}
2\Ll| \int_t^{+\infty} \Ll[\Ll(\frac s t-1  \Rr) \wedge 1\Rr] \la f, e^{-s\L} f \ra \, \d s\Rr| \le 2 \int_t^{+\infty} s^{-\al} \, \d s \le C t^{-(\al - 1)},
\end{align}
as announced.
\end{proof}
The one-step extrapolated scheme described in Proposition~\ref{p.improved} is therefore more efficient than the naive scheme described in Proposition~\ref{p.naive} whenever $\al \ge 2$. To make this point more precise, we may measure the computational complexity of each method in terms of its ``total simulated time'': if an algorithm samples $n$ trajectories, each respectively up to time $t_1, \ldots, t_n$, then the total simulated time of the algorithm is $\sum_{i = 1}^n t_i$. For any fixed $\al > 2$, we compute the total simulated time required for each method to produce an output of the order of $\de \in (0, 1]$ away from $\sigma^2(f)$. For the naive method, we need to choose $N \simeq \de^{-2}$ and $t \simeq \de^{-1}$, so that the total simulated time is of the order of $\de^{-3}$. For the one-step extrapolated method, we choose instead $t \simeq \de^{-\frac 1 {\al  - 1}}$, still with $N \simeq \de^{-2}$, so that the total simulated time is of the order of $
\de^{-\Ll( 2 + \frac 1 {\al - 1} \Rr) } \ll \de^{-3}$. If we replace the polynomial convergence to equilibrium in \eqref{e.assump.decay} by an exponential one, then one can show that the one-step extrapolated method requires only of the order of $\de^{-2} \log(\de^{-1})$ of simulated time, while the naive method still requires a total simulated time of the order of $\de^{-3}$. It is striking that a single-step extrapolation suffices to get rid of ``saturation'' phenomena as seen in Proposition~\ref{p.naive} at all orders at once.

%
%
%
%
%
%

\section{Variants of the Green-Kubo formula}
\label{s.variants.gk}

In this section, we prove the non-reversible version of the variant of the Green-Kubo formula presented in \eqref{e.induc.Poisson}. We then discuss the convergence rate of the series under the additional assumption of \eqref{e.assump.decay}.

\begin{theorem}
\label{t.pgk}
Let $f \in L^2(\P)$ be such that \eqref{e.hyp.int} holds, let $(\mu_k)_{k \in \N}$ be a sequence of real numbers in $(0,1]$, and define recursively
\begin{equation}  
\label{e.def.fk}
\Ll\{
\begin{aligned} 
& f_{-1} = f^*_{-1} := f, \\
& \forall k \in \N, \quad f_{k} := \mu_{k} (\mu_k + \L)^{-1} f_{k-1}  \quad \text{and} \quad f^*_{k} := \mu_{k} (\mu_k + \L^*)^{-1} f^*_{k-1} .
\end{aligned}
\Rr.
\end{equation}
For every $n \in \N$, the quantity
\begin{equation}  
\label{e.def.rest}
\la f_n^*, \L^{-1} f_n \ra := \lim_{\lambda \to 0} \la f_n^*, (\lambda + \L)^{-1} f_n \ra
\end{equation}
is well-defined and finite, and 
\begin{equation}  
\label{e.lim.rest}
\lim_{n \to \infty} \la f_n^*, \L^{-1} f_n \ra = 0.
\end{equation}
Moreover, for every $n \in \N$, we have
\begin{equation}  
\label{e.pgk}
\la f, \L^{-1} f \ra = \sum_{k = 0}^{n} \mu_{k}^{-1} \Ll(\la f^*_{k-1}, f_k \ra + \la f^*_k, f_k \ra\Rr) + \la f_n^*, \L^{-1} f_n \ra.
\end{equation}
\end{theorem}
\begin{proof}
Let $\lambda > 0$. From the formula \eqref{e.def.R}, we see that if $T_\lambda$ is an exponential random variable of parameter $\lambda$, independent of any other quantity in the problem, then we can represent the resolvent operator $R_\lambda$ as
\begin{equation*}  
\lambda R_\lambda g = E \Ll[ e^{-T_\lambda \L} g \Rr] ,
\end{equation*}
where $E$ denotes the expectation over the random variable $T_\lambda$. By induction, we see that if $S_n$ has the law of a sum of $(n+1)$ independent exponential random variables of parameter $\mu_0,\ldots, \mu_n$ respectively, then
\begin{equation*}  
f_n = E \Ll[ e^{-S_n \L} f \Rr].
\end{equation*}
Here again the random variable $S_n$ is independent of any other quantity in the problem, and $E$ computes the average of this random variable. Denoting by $S_n'$ an independent copy of $S_n$, we have
\begin{align*}  
\la f_n^*, (\lambda + \L)^{-1} f_n \ra & = \int_0^{+\infty} e^{-\lambda t} E \Ll[ \la e^{-S_n' \L^*} f, e^{-(S_n + t) \L} f \ra \Rr]  \, \d t \\
& = \int_0^{+\infty} e^{-\lambda t} E \Ll[ \la f, e^{-(S_n + S_n' + t)\L} f \ra \Rr] \, \d t.
\end{align*}
The assumption of \eqref{e.hyp.int} ensures that for each fixed $s \ge 0$,
\begin{equation*}  
\int_0^{+\infty} \Ll| \la f, e^{-(s+t)\L} f \ra \Rr| \, \d t \le \int_0^{+\infty}  \Ll| \la f, e^{-t\L} f \ra \Rr| \, \d t < +\infty.
\end{equation*}
Hence, the quantity
\begin{equation*}  
\int_0^{+\infty} e^{-\lambda t} \la f,e^{-(s+t)\L} f \ra \, \d t
\end{equation*}
is bounded uniformly over $s \ge 0$ and $\lambda > 0$, and converges to
\begin{equation*}  
\int_0^{+\infty} \la f, e^{-(s+t) \L} f \ra \, \d t
\end{equation*}
as $\lambda$ tends to $0$. By the dominated convergence theorem, this implies that 
\begin{equation*}  
\la f_n^*, (\lambda + \L)^{-1} f_n \ra
\end{equation*}
converges as $\lambda$ tends to infinity to 
\begin{equation*}  
\int_0^{+\infty} E \Ll[ \la f, e^{-(S_n + S_n' + t)\L} f \ra \Rr] \, \d t \in \R.
\end{equation*}
We now verify that the integral above tends to $0$ as $n$ tends to infinity. It is clear that we can realize the sequences $S_n, S_n'$ such that $S_{n+1} \ge S_n$ and $S'_{n+1} \ge S_n'$, and moreover, since the sequence $(\mu_k)$ is bounded by $1$, we have that $S_n$ and $S_n'$ tend to infinity almost surely. For each $M \ge 1$, we can bound the absolute value of the integral above by
\begin{equation}  
\label{e.easy.decomp}
P \Ll[ S_n + S_n' \le M \Rr] \int_0^{+\infty} \Ll| \la f, e^{-t\L} f \ra \Rr| \, \d t + \int_M^{+\infty} \Ll| \la f, e^{-t\L} f \ra \Rr| \, \d t.
\end{equation}
By the assumption of \eqref{e.hyp.int}, we may fix $M$ sufficiently large to make the second integral in \eqref{e.easy.decomp} as small as desired. Then taking $n$ sufficiently large ensures that the probability on the left of \eqref{e.easy.decomp} is arbitarily small. This completes the proof of \eqref{e.lim.rest}.

\smallskip

There remains to show \eqref{e.pgk}. For each $\lambda > 0$, we introduce the shorthand notation
\begin{equation*}  
\mu_k^\lambda := \mu_k-\lambda.
\end{equation*}
Applying the resolvent formula \eqref{e.resolvent} recursively, we see that
\begin{equation}  
\label{e.induc.resolv}
R_\lambda = \Ll(\sum_{k = 0}^n \mu^\lambda_0 \cdots \mu^\lambda_{k-1} R_{\mu_0} \cdots R_{\mu_k}\Rr) + \mu^\lambda_0\cdots \mu^\lambda_n \, R_{\mu_0} \cdots R_{\mu_n} R_\lambda.
\end{equation}
We define recursively
\begin{equation*}  
f_{-1}^{\lambda} = f_{-1}^{\lambda, *} := f,
\end{equation*}
\begin{equation*}  
\forall k \in \N, \quad f^{\lambda}_{k} := \mu^\lambda_{k} R_{\mu_k} f^\lambda_{k-1}  \quad \text{and} \quad f^{\lambda,*}_{k} := \mu^\lambda_{k} R_{\mu_k}^* f^{\lambda,*}_{k-1} .
\end{equation*}
Using the formula \eqref{e.induc.resolv} with the sequence $(\mu_k)_{k \in \N}$ replaced by $(\mu_0, \mu_0, \mu_1, \mu_1, \ldots)$, we get that
\begin{equation}  
\label{e.decomp.lambda}
\la f, R_\lambda f \ra = \sum_{k = 0}^n \frac{1}{\mu^\lambda_k} \Ll( \la f^{\lambda, *}_{k-1}, f^\lambda_k \ra  + \la f^{\lambda, *}_k, f^\lambda_k\ra \Rr)  + \la f^{\lambda, *}_n, R_\lambda f^\lambda_n \ra,
\end{equation}
where for simplicity, we may assume that $\lambda > 0$ is sufficiently small that $\mu_k^\lambda \neq 0$ for each $k \le n$. Note that $f_n^\lambda$ differs from $f_n$ by a scalar factor only, and this factor tends to $1$ as $\lambda$ tends to $0$; and similarly for $f_n^{\lambda, *}$. It is therefore clear that
\begin{equation*}  
\lim_{\lambda \to 0} \la f^{\lambda, *}_n, R_\lambda f^\lambda_n \ra = \la f^*_n, \L^{-1} f_n \ra.
\end{equation*}
Passing to the limit $\lambda \to 0$ in the sum on the right side of \eqref{e.decomp.lambda} is similar (only simpler), and we obtain \eqref{e.pgk}.
\end{proof}

We now give some simple criteria for evaluating the size of the remainder term in the decomposition~\ref{e.pgk}. This will depend on the sequence $(\mu_k)$ we choose. The simplest case is when we choose this sequence to be constant.
\begin{proposition}[Remainder estimate, constant step size]
\label{p.constant.step}
For each $\al > 1$, there exists a constant $C < \infty$ such that the following holds. Let $f \in L^2(\P)$ satisfy \eqref{e.assump.decay}, and let $(f_k)_{k \in \N}$, $(f^*_k)_{k \in \N}$ be defined according to \eqref{e.def.fk} with the choice $\mu_k \equiv 1$. For every $n \in \N$, we have
\begin{equation}  
\label{e.constant.step}
|\la f_n^*, \L^{-1} f_n \ra | \le C (1+n)^{-(\al-1)}.
\end{equation}
\end{proposition}
\begin{proof}
Denote by $S_n$ and $S_n'$ two idependent random variables, each having the law of a sum of $n$ independent exponential random variables of parameter $1$. We assume that these random variables are also independent of any other quantity discussed so far, and denote by $E$ the expectation with respect to these random variables. We have
\begin{equation*}  
f_n = E \Ll[ e^{-S_n \L} f \Rr]  \quad \text{and} \quad f_n^* = E \Ll[ e^{-S_n' \L^*} f \Rr]  .
\end{equation*}
Moreover,
\begin{align*}  
\la f_n^*, \L^{-1} f_n \ra & = \int_0^{+\infty} \la f^*_n, e^{-t\L} f_n \ra \, \d t \\
& = \int_0^{+\infty} E \Ll[ \la e^{-S_n'\L^*} f, e^{-(t+S_n) \L} f \ra \Rr]  \, \d t \\
& = \int_0^{+\infty} E \Ll[ \la f, e^{-(t+ S_n + S_n')\L} f \ra \Rr] \, \d t.
\end{align*}
By an elementary large deviation estimate, there exists a constant $C < \infty$ such that 
\begin{equation*}  
P \Ll[ \Ll|S_n + S_n' - 2n\Rr| \le n \Rr] \le C \exp \Ll( - n/C \Rr) .
\end{equation*}
Using the assumption \eqref{e.assump.decay}, we thus get that 
\begin{align*}  
\big| \la f_n^*, \L^{-1} f_n \ra \big| & \le \int_0^{+\infty} E\Ll[(1+t+S_n + S_n')^{-\al}\Rr] \, \d t \\
& \le C \exp(-n/C) + \int_0^{+\infty} (1+t+n)^{-\al} \, \d t \\
& \le C (1+n)^{-(\al - 1)},
\end{align*}
as announced.
\end{proof}
Under the assumption \eqref{e.assump.decay}, it is also possible to use geometrically increasing step sizes with no degradation of the convergence rate, as the next proposition demonstrates. This can be interesting, since it reduces considerably the number of Poisson equations one needs to solve: for a desired precision of $\de > 0$, using geometrically increasing step sizes requires only the resolution of a constant times  $\log \de^{-1}$ such equations, as opposed to $\de^{-1/(\al-1)}$ with the constant-step-size scheme of Proposition~\ref{p.constant.step}. (On the other hand, if the assumption \eqref{e.assump.decay} was replaced by exponentially fast convergence, then the constant-step-size scheme should be preferred.)
\begin{proposition}[Remainder estimate, geometric step size]
\label{p.geom.step}
For each $\al > 1$, there exists a constant $C < \infty$ such that the following holds. Let $f \in L^2(\P)$ satisfy \eqref{e.assump.decay}, and let $(f_k)_{k \in \N}$, $(f^*_k)_{k \in \N}$ be defined according to \eqref{e.def.fk} with the choice $\mu_k \equiv 2^{-k}$. For every $n \in \N$, we have
\begin{equation}  
\label{e.geom.step}
|\la f_n^*, \L^{-1} f_n \ra | \le C 2^{-n(\al-1)}.
\end{equation}
\end{proposition}
\begin{remark}  
In both \eqref{e.constant.step} and \eqref{e.geom.step}, the right-hand side is a multiple of $(\mu_0^{-1} + \cdots + \mu_n^{-1})^{-(\al-1)}$. 
\end{remark}
\begin{proof}
Let $(T_k)_{k \in \N}$ be independent exponential random variables of respective parameter $(2^{-k})_{k \in \N}$, and for each $n \in \N$, let
\begin{equation}  
\label{e.def.Sn}
S_n = \sum_{k = 0}^n T_k.
\end{equation}
We assume that the random variables $(T_k)$ are independent of the other quantities of our problem, and denote by $E$ the expectation with respect to these random variables. We also give ourselves $S_n'$ an independent copy of $S_n$, and recall that
\begin{equation*}  
f_n = E \Ll[ e^{-S_n \L} f \Rr] \quad \text{and} \quad f_n^* = E \Ll[ e^{-S_n\L^*} f \Rr] .
\end{equation*}
We decompose  the rest of the proof into three steps.

\smallskip

\emph{Step 1.} In this first step, we show that for each $K < \infty$, there exists a constant $C < \infty$ such that for every $k \le n$,
\begin{equation}  
\label{e.bound.Sn}
P\Ll[S_n \le 2^k\Rr] \le C \exp \Ll( -K(n-k) \Rr) .
\end{equation}
By Chebyshev's inequality, we have
\begin{align*}  
P\Ll[S_n \le 2^k\Rr] & \le \E \Ll[ \exp \Ll(1 -2^{-k} \Ll( T_1 + \cdots + T_n \Rr)  \Rr)   \Rr]  \\
& \le \exp \Ll( 1 - \sum_{j = 0}^n \log \Ll( 1+2^{j-k} \Rr)  \Rr).
\end{align*}
We select $k_0(K) < \infty$ such that 
\begin{equation*}  
\log(1+2^{k_0}) \ge K,
\end{equation*}
and deduce that
\begin{equation*}  
P\Ll[S_n \le 2^k\Rr] \le \exp \Ll( 1 - K (n-k-k_0)_+ \Rr) .
\end{equation*}
This shows \eqref{e.bound.Sn}.

\smallskip

\emph{Step 2.} We show that for every $\be \in (0,\infty)$, there exists a constant $C < \infty$ such that for ever $n$,
\begin{equation}  
\label{e.Sn.moment}
E \Ll[ (1+S_n)^{-\beta} \Rr] \le C 2^{-n\beta }.
\end{equation}
We rewrite the expectation on the left side above as
\begin{equation*}  
E \Ll[ (1+S_n)^{-\beta} \Rr] = \beta \int_1^{+\infty} t^{-(\beta + 1)} P \Ll[ 1+S_n \le t \Rr] \, \d t,
\end{equation*}
and decompose the integral into several parts. The easiest part is 
\begin{equation*}  
\beta \int_{2^n}^{\infty} t^{-(\beta + 1)} P \Ll[ 1+S_n \le t \Rr] \, \d t \le \beta \int_{2^n}^\infty t^{-(\beta + 1)} \, \d t \le 2^{-n\beta}.
\end{equation*}
We split the remaining part along a dyadic sequence:
\begin{align*}
\int_1^{2^n} s^{-(\beta+1)} P[1+S_n \le s] \, \d s & \le \sum_{k = 0}^{n-1} \int_{2^k}^{2^{k+1}} s^{-(\beta+1)} P[1+S_n \le s] \, \d s  \\
& \le \sum_{k = 0}^{n-1} 2^{-k\beta} P[S_n \le 2^{k+1}].
\end{align*}
Using the result of the first step with the choice of $K = (\beta+1) \log 2$, we obtain that the sum above is bounded by
\begin{equation*}  
C \sum_{k = 0}^{n-1} 2^{-k\beta} \, 2^{-(n-k)(\beta+1)} \le C 2^{-n\beta}.
\end{equation*}
This completes the proof of \eqref{e.Sn.moment}.

\smallskip

\emph{Step 3.} We complete the proof. As in the proofs of Theorem~\ref{t.pgk} and Proposition~\ref{p.constant.step}, our starting point is to observe that
\begin{equation*}  
\la f_n^*, \L^{-1} f_n \ra = \int_0^{+\infty} E \Ll[ \la f, e^{-(t+S_n + S_n') \L } f \ra \Rr] \, \d t.
\end{equation*}
Using \eqref{e.assump.decay}, Fubini's theorem and a change of variables, we infer that
\begin{align*}  
\big|\la f_n^*, \L^{-1} f_n \ra \big| & \le \int_0^{+\infty} E\Ll[(1+t+S_n + S_n')^{-\al}\Rr] \, \d t \\
& \le \int_0^{+\infty} E\Ll[(1+t+S_n)^{-\al}\Rr] \, \d t  \\
& \le (\al - 1)^{-1} E \Ll[(1+S_n)^{-(\al - 1)}  \Rr] .
\end{align*}
The conclusion then follows from \eqref{e.Sn.moment}.
\end{proof}
Propositions~\ref{p.constant.step} and \ref{p.geom.step} provide us with guidelines for the choice of the terminal index $n$ in the series expansion of $\la f, \L^{-1} f \ra$ given in \eqref{e.pgk}. One central feature of this series expansion is the fact that as the index $k$ of the series increases and $f_k$ and $f_k^*$ become more difficult to compute, these quantities also become smaller, and we therefore need much fewer samples to evaluate the average $\la f_k^*, f_k \ra$ accurately. The smallness of $f_k$ is most clearly seen if we specialize the proposition below to the reversible case.
\begin{proposition}
\label{p.size.fk}
There exists a constant $C < \infty$ such that the following holds.  Let $f \in L^2(\P)$ satisfy \eqref{e.assump.decay}, and let $(f_k)_{k \in \N}$, $(f_k^*)_{k \in \N}$) be defined according to \eqref{e.def.fk} with the choice $\mu_k \equiv 1$. For every $k \in \N$, we have
\begin{equation*}  
\Ll|\la f_{k-1}^*, f_k \ra \Rr| + \Ll|\la f_k^*, f_k \ra\Rr| \le C (1+k)^{-\al}.
\end{equation*}
If instead, the sequences $(f_k)_{k \in \N}$ and $(f_k^*)_{k \in \N}$ are defined according to \eqref{e.def.fk} with the choice $\mu_k \equiv 2^{-k}$, then for every $k \in \N$,
\begin{equation*}  
\Ll|\la f_{k-1}^*, f_k \ra \Rr| + \Ll|\la f_k^*, f_k \ra\Rr|  \le C 2^{-k\al}.
\end{equation*}
\end{proposition}
The proof of this proposition is almost identical to the proofs of Propositions~\ref{p.constant.step} and \ref{p.geom.step}, so we omit it. The point we wish to stress here is that at least under the assumption of reversibility (that is, $\L^* = \L$ and thus $f_k^* = f_k$), the typical size of $f_k$ (as measured in $L^2(\P)$) is small: it is of the order of $k^{-\al/2}$ for constant step sizes, and of the order of $2^{-k\al/2}$ for geometrically increasing step sizes. 

\smallskip

In the non-reversible setting, the situation is more subtle. Indeed, the assumption of \eqref{e.assump.decay} no longer informs us on the size of $e^{-t\L} f$ in $L^2(\P)$; similarly, we can no longer sharply estimate the size of $f_k$ (whatever the choice of $\mu_k$) from this assumption. In Proposition~\ref{p.decay.discrete} below, we consider dynamical systems (which in some sense is the ``worst-case scenario'' for these bounds) and show, under an assumption comparable to \eqref{e.assump.decay}, that $f_k$ is of the order of $k^{-\frac 1 4}$ in the case of $\mu_k \equiv 1$, and that $f_k$ is of the order of $2^{-\frac k 2}$ in the case of geometrically increasing step sizes. This gives further motivation for using geometrically increasing step sizes.

\section{Discrete-time Markov chains}
\label{s.dynsys}

In this section, we explain how to adapt the arguments of the previous section to the setting of discrete-time Markov chains. We then discuss the rate of decay to zero of the remainder in the series expansion \eqref{e.pgk} in the ``extreme'' case of dynamical systems. 

\smallskip

Let $(X_n)_{n \in \N}$ be a discrete-time Markov chain on $\X$, and let $P$ be its transition operator. That is, with hopefully transparent notation,
\begin{equation*}  
Pf(x) := \EE_x \Ll[ f(X_1) \Rr].
\end{equation*}
We define
\begin{equation}  
\label{e.newL}
\L := \mathrm{Id} - P.
\end{equation}
We assume that the measure $\P$ is invariant for $(X_n)$, and continue denoting by~$\la \cdot, \cdot \ra$ the scalar product in $L^2(\P)$. We fix $f \in L^2(\P)$ of mean zero. As in the continuous case, there is no simple and general theorem asserting whether the additive functional 
\begin{equation}  
\label{e.additive.discrete}
\frac 1 {\sqrt{n}} \sum_{k = 0}^n f(X_k)
\end{equation}
converges in law to a Gaussian as $n$ tends to infinity. Note that at this level of generality, the question covers dynamical systems as a particular case; we refer to \cite{liver, dolg, melbourne} and references therein for results in this context. In analogy with the continuous case, we assume throughout that
\begin{equation} 
\label{e.hyp.int.dynsys}
\sum_{k = 0}^{+\infty} \big| \la f, P^k f \ra \big| < \infty.
\end{equation}
Under some additional assumptions (e.g.\ reversibility is sufficient \cite{kipvar}), the random variable in \eqref{e.additive.discrete} converges in law to a centered Gaussian of variance 
\begin{align}  
\notag
\sigma^2(f) & :=  \la f, f \ra + 2 \lim_{\lambda \to 0} \la f, (\lambda + \L)^{-1} f \ra \\
\label{e.identity.resolv.discrete}
& = - \la f, f \ra + 2\sum_{k = 0}^{+\infty} \la f, P^k f \ra .
\end{align}
The analogue of the formula \eqref{e.gk2}, adapted to the non-reversible and discrete-time setting, reads
\begin{equation}  
\label{e.gk2.discrete}
\sigma^2(f) = -\la f, f \ra + 2 \sum_{k = 0}^{+\infty} \big( \la (P^*)^k f, P^k f \ra + \la (P^*)^k f, P^{k+1} f \ra \big)  ,
\end{equation}
where $P^*$ is the adjoint of $P$ in $L^2(\P)$. This formula can be used directly to evaluate $\sigma^2(f)$ (this is essentially the formula used in Section~\ref{s.discrete} below). We also have that the exact statement of Theorem~\ref{t.pgk}, with $\L$ now defined by \eqref{e.newL}, holds---and the proof given there applies with essentially no change besides a minor adaptation to the discrete-time setting. The minor adaptation boils down to replacing the formula for the resolvent in \eqref{e.def.R} with
\begin{equation}
\label{e.alt.R}
R_\lambda = \sum_{k = 0}^{+\infty} (\lambda + 1)^{-(k+1)} \, P^k.
\end{equation}

\smallskip

Strengthening \eqref{e.hyp.int.dynsys}, and in analogy with \eqref{e.assump.decay}, we may assume that there exists an exponent $\al > 1$ such that for every $k \in \N$, 
\begin{equation}  
\label{e.decay.disc}
\big| \la f, P^k f \ra \big| \le (1+k)^{-\al}.
\end{equation}
This assumption implies the validity of Propositions~\ref{p.constant.step}, \ref{p.geom.step} and \ref{p.size.fk}. The proofs are almost identical except for the use of \eqref{e.alt.R} in place of \eqref{e.def.R}. 

\smallskip 

With reversible dynamics, the assumption of \eqref{e.decay.disc} directly translates into information on the size of $P^k f$ itself. This is not so in general. We can illustrate this phenomenon most clearly in the particular case of dynamical systems. Indeed, assume that there exists a measure-preserving mapping $\T: \X \to \X$ such that for every $n$, $X_n = \T^n(X_0)$, or equivalently,
\begin{equation}  
\label{e.def.dynsys}
P f = f \circ \T. 
\end{equation}
Then the $L^2$ norm of $P^k f$ does not depend on $k$, since $\T$ is measure-preserving. In this case, we thus need many samples of $\T^k f$ and $(\T^*)^k f$ if we wish to estimate the average $\la (\T^*)^k f, \T^k f \ra$ using independent draws. On the other hand, we may expect the quantity $f_k$ appearing in Theorem~\ref{t.pgk} to display some cancellations that make it smaller. The next proposition quantifies this more precisely.
\begin{proposition}
\label{p.decay.discrete}
There exists a constant $C < \infty$ such that the following holds.  Let $P$ be the transition operator of a dynamical system, see \eqref{e.def.dynsys}, let $f \in L^2(\P)$ satisfy \eqref{e.decay.disc}, and let $(f_k)_{k \in \N}$, $(f_k^*)_{k \in \N}$ be defined according to \eqref{e.def.fk} with the choice $\mu_k \equiv 1$. For every $k \in \N$, we have
\begin{equation}  
\label{e.disc1}
\Ll|\la f_{k}^*, f_k^* \ra \Rr| + \Ll|\la f_k, f_k \ra\Rr| \le C (1+k)^{-\frac 1 2}.
\end{equation}
If instead, the sequences $(f_k)_{k \in \N}$ and $(f_k^*)_{k \in \N}$ are defined according to \eqref{e.def.fk} with the choice $\mu_k \equiv 2^{-k}$, then for every $k \in \N$,
\begin{equation}  
\label{e.disc2}
\Ll|\la f_{k}^*, f_k^* \ra \Rr| + \Ll|\la f_k, f_k \ra\Rr|  \le C 2^{-k}.
\end{equation}
\end{proposition}
We stress again that the exact statement of Proposition~\ref{p.size.fk} is true under the assumption of \eqref{e.decay.disc}. Proposition~\ref{p.decay.discrete} is a statement about the respective sizes of $f_k^*$ and $f_k$, \emph{not} about the size of the correlation between these quantities.
\begin{proof}[Proof of Proposition~\ref{p.decay.discrete}]
Let $T_1, T_2, \ldots$ be independent random variables following the same law of geometric type:
\begin{equation}
\label{e.def.geom}
\forall k \in \N, \quad P \Ll[ T_i = k \Rr] = \mu_i (\mu_i + 1)^{-(k+1)} .
\end{equation}
We assume that these random variables are independent of the other randomness of the problem, and denote by $E$ the expectation with respect to these variables. We also set, for each $k \in \N$,
\begin{equation*}  
S_k := T_1 + \cdots + T_k,
\end{equation*}
and give ourselves an independent copy $S_k'$ of $S_k$. In view of the formula \eqref{e.alt.R}, we have the representation
\begin{equation*}  
f_k = E \Ll[ P^{S_k} f \Rr] ,
\end{equation*}
and thus
\begin{equation*}  
\la f_k, f_k \ra = E \Ll[ \la P^{S_k} f, P^{S_k'} f \ra \Rr] .
\end{equation*}
Since the mapping $\T$ is measure-preserving, we deduce that
\begin{equation*}  
\la f_k, f_k \ra = E \Ll[ \la f, P^{|S_k - S_k'|} f \ra \Rr] ,
\end{equation*}
and by the assumption of \eqref{e.decay.disc} and the triangle inequality, we get
\begin{align*}  
\big| \la f_k, f_k \ra \big| & \le E \Ll[ (1+|S_k - S_k'|)^{-\al} \Rr] \\
& \le C E \Ll[ (1+ |S_k - E[S_k]|)^{-\al} \Rr] .
\end{align*}
We now use the decomposition
\begin{equation}  
\label{e.fubini.friend}
E \Ll[ (1+ |S_k - E[S_k]|)^{-\al} \Rr] = \int_1^\infty \frac{\al}{s^{\al + 1}} \,  P \Ll[ 1 + |S_k - E[S_k]| \le s \Rr] \, \d s.
\end{equation}
We first analyze the case of constant step size, that is $\mu_k \equiv 1$. In this case, the random variable $S_k$ is a sum of independent and identically distributed random variables. Denote by $\tau^2$ the variance of $T_1$. By the Berry-Esseen theorem, see \cite[Theorem~5.5]{petrov}, there exists a constant $C < \infty$ such that for every $k \in \N$ and $s \in \R$,
\begin{equation*}  
\Ll|P\Ll[ \Ll| S_k - E[S_k] \Rr| \le s\sqrt{k}\Rr] - \frac {2}{\tau\sqrt{2\pi}} \int_{0}^s \exp \Ll( -\frac{x^2}{2\tau^2}  \Rr) \, \d x   \Rr|  \le \frac{C}{\sqrt{k}}.
\end{equation*}
We thus obtain that 
\begin{align*}  
E \Ll[ (1+ |S_k - E[S_k]|)^{-\al} \Rr] & \le \frac{C}{\sqrt k}  + C \int_1^\infty \frac{1}{s^{\al + 1}} \int_{0}^{\frac s {\sqrt{k}}} \exp \Ll( -\frac{x^2}{2\tau^2} \Rr) \, \d x \, \d s \\
& \le \frac{C}{\sqrt{k}} + C \int_1^{\infty} \frac 1 {s^{\al + 1}} \, \frac{s}{\sqrt{k}} \, \d s \\
& \le \frac{C}{\sqrt{k}},
\end{align*}
as announced.

\smallskip

We now turn to the case of geometrically increasing step sizes, that is, $\mu_k \equiv 2^{-k}$. We first show that there exists a constant $C < \infty$ such that for every $k \in \N$ and $s \ge 0$,
\begin{equation}  
\label{e.easy.decay}
P\Ll[|S_k - E[S_k]| \le s\Rr] \le C 2^{-k} \Ll({s+1}\Rr).
\end{equation}
In order to do so, it suffices to verify that there exists $C < \infty$ such that for every $k \in \N$ and $x \in \N$,
\begin{equation*}  
P \Ll[ S_k = x \Rr] \le C 2^{-k}.
\end{equation*}
It follows from the definition of $T_k$ that this property is satisfied if $S_k$ is replaced by $T_k$. By independence between $T_k$ and $S_{k-1}$, we deduce
\begin{align*}  
P \Ll[ S_k = x \Rr] & = E \Ll[ P \Ll[ T_k = x - S_{k-1} \ | \ S_{k-1} \Rr] 
 \Rr] \\
 & \le E \Ll[ \sup_{z \in \Z} P\Ll[T_k = z \ | \ S_{k-1} \Rr] \Rr] \\
 &  \le C 2^{-k}.
\end{align*}
This implies \eqref{e.easy.decay}. Combining this with \eqref{e.fubini.friend} yields
\begin{equation*}  
E \Ll[ (1+ |S_k - E[S_k]|)^{-\al} \Rr]  \le C 2^{-k} \int_1^\infty \frac{1}{s^{\al }} \, \d s \le C 2^{-k},
\end{equation*}
completing the proof.
\end{proof}
\begin{remark}  
One can verify that the estimates obtained in Proposition~\ref{p.decay.discrete} are sharp, in the sense that assuming the converse inequality to \eqref{e.decay.disc}, one can show the converse inequalities to \eqref{e.disc1} and \eqref{e.disc2}, up to multiplicative constants.
\end{remark}

\section{Examples}
\label{s.examples}

In the second part of the paper, we explore thoroughly the example of a reversible random walk in a random environment, or equivalently the homogenization of discrete divergence-form operators with random coefficients. Under some assumptions, the relaxation property \eqref{e.assump.decay} holds with $\al = \frac d 2 + 1$. We will explain how to implement numerical schemes derived from Theorem~\ref{t.pgk} by leveraging on the mixing properties of the functions $f_k$ under spatial translations. We will show that such numerical schemes achieve essentially optimal complexity, requiring $C \de^{-2} \log (\de^{-1})$ computations to compute the effective diffusivity at precision $\de > 0$.

\smallskip

An alternative route to solve the same problem would be to use Monte-Carlo methods. Computing the effective diffusivity at precision $\delta > 0$ using the method described in Proposition~\ref{p.improved} would call for the simulation of $\delta^{-2}$ random walk trajectories that run up to time $\delta^{-\frac 2 d}$, bringing the total cost of the method to a suboptimal $\delta^{-2-\frac 2 d}$. (See \cite[Theorem~4.1]{MCdisc} for the verification that a large deviation estimate much stronger than \eqref{e.fourth} holds.)

\smallskip

In this section, we discuss briefly other examples of interest where the methods described here can be applied. 

\subsection{Tagged particle in exclusion process}
One example where the approach based on Theorem~\ref{t.pgk} seems very costly to implement is the problem of computing the effective diffusivity of a tagged particle in a system of particles evolving according to, say, exclusion dynamics on $\Z^d$. We refer to \cite{lig2} for the definition of this and other models of interacting particle systems. Indeed, the Poisson equations appearing in \eqref{e.def.fk} are posed in $\{0,1\}^{\Z^d}$. Even if we restrict the dynamics to a finite box $B \subset \Z^d$, the size of each ergodic component of the dynamics is exponentially large as a function of the volume of $B$, which seems to preclude any computationally efficient approach to the resolution of these equations. In contrast, the simulation of the dynamics on $B$ is feasible, and by exchangeability, the trajectory of each of the simulated particles can be used as a sample of the trajectory of a tagged particle. One would then need to assert precisely the independence properties of these trajectories, but it is clear that the diffusivity can be computed in polynomial time using an MCMC scheme, and possibly in about $\de^{-2-\frac 2 d}$ operations.

\subsection{Langevin dynamics}

Hypoelliptic dynamics can also be considered. We explore briefly here the case of Langevin dynamics in a random potential. Let $\phi$ be a smooth stationary and ergodic function over $\Rd$. We consider the Markov process $(X_t, V_t)_{t \ge 0}$, taking values in $\Rd \times \Rd$,  whose infinitesimal generator is given by
\begin{equation}
\label{e.def.LL}
L f := \Delta_vf - v \cdot \nabla_vf - \nabla_x \phi \cdot \nabla_v f + v \cdot \nabla_x.
\end{equation}
Note that in this case, the process $(X_t)$ itself is an additive functional, since 
\begin{equation*}  
X_t = \int_0^t V_s \, \d s,
\end{equation*}
while for random walks or tagged particles the connection between the dynamics and additive functionals is less immediate (but see \cite{kipvar}).
The fact that $(\eps X_{\eps^{-2} t})_{t \ge 0}$ converges in law to Brownian motion as $\eps$ tends to zero was proved in \cite{papvar}. Moreover, the invariant measure for the process of the environment as seen from $X$ is explicit. Under suitable assumptions on the mixing properties of $\phi$, we may expect, in analogy with the simpler setting explore in the second part of the paper, that the polynomial relaxation \eqref{e.assump.decay} holds with $\al = \frac d 2 + 1$. 

\smallskip

It can be relatively costly (but not as prohibitive as for interacting particle systems) to compute approximate solutions to the Poisson equations appearing in \eqref{e.def.fk}. Compared with the situation encountered for random walks, the additional difficulty stems in part from the fact that the operator is hypoelliptic, but also simply from the fact that there are more variables to solve for (due to the addition of velocity coordinates). The superiority of this approach compared to that based on Monte-Carlo sampling is thus unclear, and probably depends on the situation. In one extreme case, if the dimension is small and~$\phi$ is a deterministic periodic function (whose periods need not be known), then an approach similar to that exposed in the second part of the paper is most probably best (and it should be preferable to choose $\mu_k \equiv 1$ there, since the convergence to equilibrium is exponentially rapid in the periodic setting). At the other extreme, in the random case and in high dimension, the Monte-Carlo approach probably becomes superior.

\chapter{Stochastic homogenization}
\label{part.two}

\section{Notation and basic properties}
\label{s.defs}

In this second part of the paper, we study the problem of approximating the homogenized matrix of divergence-form operators with random coefficients, as stated in the introduction. In the present section, we set up notation and assumptions and recall basic facts of homogenization. We then proceed to prove the error estimate \eqref{e.main} in Section~\ref{s.error}. In Section~\ref{s.compl}, we complete the proof of Theorem~\ref{t.main} by estimating the computational complexity of the method, and prove the complexity lower bound given by Proposition~\ref{p.lower.bound}. We report on the results of numerical tests in Section~\ref{s.numer}. Finally, we explore a variant of the method based on discrete-time parabolic equations in Section~\ref{s.discrete}.

\smallskip

We focus on the discrete space $\Z^d$, $d \ge 2$. We say that $x, y \in \Zd$ are neighbors, and write $x \sim y$, if $|x-y| = 1$. This turns $\Z^d$ into a graph, and we denote by~$\B$ the corresponding set of undirected edges. For a fixed ellipticity parameter $\Lambda \in (1,\infty)$, we let $\Omega = [1,\Lambda]^{\B}$ denote the set of coefficient fields, equipped with the product $\sigma$-algebra. We denote by $\a = (\a_e)_{e \in \B}$ the canonical random variable on $\Omega$. We let $\P$ be a probability measure on $\Omega$ under which the random variables $(\a_e)_{e \in \B}$ are independent and identically distributed (i.i.d.), and denote by $\E$ the associated expectation. Besides uniform ellipticity, this independence assumption is our main quantitative hypothesis. 
%
%
The group $\Z^d$ acts on $\Omega$ by translations via
\begin{equation*}  
\Ll(\theta_x \a\Rr)_{e} = \a_{x + e},
\end{equation*}
where $x + e$ denotes the edge $e$ translated by the vector $x$. 
Let $(\e_1,\ldots,\e_d)$ be the canonical basis of $\Z^d$. We may identify each vector $\e_i$ with the corresponding edge $\{0,\e_i\}$ when convenient. The discrete gradient $\nabla$ acts on functions $f : \Zd \to \R$ according to
\begin{equation*}  
\nabla f(x) := \Ll( f(x + \e_1) - f(x), \ldots, f(x+\e_d) - f(x)\Rr) ,
\end{equation*}
while the discrete divergence acts on function $F = (F_1,\ldots,F_d) : \Omega \to \Rd$ via
\begin{equation*}  
\nabla \cdot F(x) := \sum_{i = 1}^d \Ll( F_i(x) - F_i(x-\e_i) \Rr) .
\end{equation*}
One can check that up to a minus sign, the gradient and divergence operators are formal adjoints of one another with respect to the counting measure on $\Zd$. The operator of interest to us is defined by
\begin{equation*}  
L := -\nabla \cdot \a \nabla,
\end{equation*}
where here we interpret $\a(x)$ as the diagonal matrix $\diag(\a_{x+\e_1}, \ldots,\a_{x+\e_d})$. More explicitly, for every $f : \Zd \to \R$,
\begin{equation*}  
Lf(x) = \sum_{y \sim x} \a_{xy} \Ll( f(x) - f(y) \Rr) .
\end{equation*}
It is useful to lift this operator to stationary functions on $\Omega$. We first define the corresponding discrete gradient $D$, which acts on functions $f : \Omega \to \R$ via
\begin{equation*}  
D f(\a) := \Ll(  
f(\theta_{\e_1}\a) - f(\a), \ldots, f(\theta_{\e_d}\a) - f(\a) 
\Rr),
\end{equation*}
and the discrete divergence, which acts on functions $F = (F_1, \ldots, F_d) : \Omega \to \Rd$ via
\begin{equation*}  
D \cdot F(\a) := \sum_{i =1}^d \Ll( F_i(\a) - F_i(\theta_{-\e_i} \a) \Rr) .
\end{equation*}
One can verify that up to a minus sign, the gradient and divergence operators are formal adjoints of one another under the measure $\P$. The ``lifted''  operator is 
\begin{equation}  
\label{e.def.mclL}
\L := -D \cdot \a D,
\end{equation}
where we understand here that $\a = \diag(\a_{\e_1}, \ldots,\a_{\e_d})$. In a more explicit form, for every $f : \Omega \to \R$,
\begin{equation*}  
\L f(\a) = \sum_{z \sim 0} \a_{0z} \Ll( f(\a) - f(\theta_z \a) \Rr) .
\end{equation*}
The operator $\L$ is self-adjoint on $L^2(\Omega)$. 
With each function $f : \Omega \to \R$, one can associate a stationary extension $\tilde f : \Omega \times \Zd \to \R$ defined by
\begin{equation*}  
\tilde f(\a,x) := f(\theta_x \a). 
\end{equation*}
Any function on $\Omega \times \Zd$ that can be written in this way is said to be stationary. 

We let $\Lpot(\Omega)$ denote the closure in $L^2(\Omega,\Rd)$ of the set of gradient fields $\{D f \ : \ f \in L^2(\Omega,\Rd)\}$. Each element of $\Lpot(\Omega)$ can be realized as the gradient of a possibly non-stationary function $\Omega \times \Zd \to \R$. Naturally, the gradient determines the function up to an additive constant. We denote by $D f$ any element of $\Lpot(\Omega)$, keeping in mind that $f$ may be a non-stationary function. We fix once and for all a unit vector $\xi \in \Rd$, and define the gradient of the corrector in the direction of $\xi$ as the unique $D\phi \in \Lpot(\Omega)$ solution to
\begin{equation}
\label{e.def.corr}
- D \cdot  \a (\xi + D \phi) = 0.
\end{equation}
The equation \eqref{e.def.corr} is interpreted in the weak sense of
\begin{equation}  \label{e.weak.corr}
\forall \, D f \in \Lpot(\Omega), \quad \E \Ll[ D f \cdot \a \Ll( \xi + D \phi \Rr)  \Rr] = 0.
\end{equation}
The existence of $D \phi$ can be showed by considering first the approximate correctors $(\phi_\mu)_{\mu > 0}$. For each $\mu > 0$, the approximate corrector $\phi_\mu \in L^2(\Omega)$ solves an equation with an additional massive term:
\begin{equation}  
\label{e.def.phimu}
\mu \phi_\mu - D \cdot \a \Ll(\xi + D \phi_\mu\Rr) = 0,
\end{equation}
and its gradient converges to $D\phi$:
\begin{equation}
\label{e.conv.phimu}
D \phi_\mu \xrightarrow[\mu \to 0]{L^2(\Omega)} D \phi.
\end{equation}
Indeed, energy estimates show that $D\phi_\mu$ is bounded in $L^2(\Omega)$, and therefore converges weakly to some element of $\Lpot(\Omega)$ up to extraction. Using also that $\mu^\frac 1 2 \phi_\mu$ is bounded in $L^2(\Omega)$, we infer that the convergence in \eqref{e.conv.phimu} holds in the weak sense. In order to justify strong convergence in $L^2(\Omega)$, it then suffices to verify that 
\begin{equation*}  
\limsup_{\mu \to 0} \E\Ll[D \phi_\mu \cdot \a D \phi_\mu\Rr] \le \E\Ll[D \phi \cdot \a D \phi\Rr].
\end{equation*}
This follows from 
\begin{equation*}  
\E \Ll[  \mu \phi_\mu^2 + D \phi_\mu \cdot \a D \phi_\mu \Rr] = -\E[D \phi_\mu \cdot \a \xi]
\end{equation*}
and the fact that the right side converges to 
\begin{equation*}  
-\E[D \phi \cdot \a \xi] = \E[D \phi \cdot \a D \phi].
\end{equation*}
As is well-known, the discrete operator $-\nabla \cdot \a \nabla$ homogenizes to the continuous differential operator $-\nabla \cdot \ahom \nabla$, where $\ahom$ is the constant-in-space, deterministic matrix characterized by the relation
\begin{equation}  \label{e.charact.ahom}
\xi \cdot \ahom \xi = \E \Ll[ (\xi + D \phi) \cdot \a (\xi + D \phi) \Rr] .
\end{equation}
Using the weak formulation \eqref{e.weak.corr}, we can rewrite this formula as
\begin{equation}
\label{e.basis}
\xi \cdot \ahom \xi = \E[\xi \cdot \a \xi] - \E\Ll[D \phi \cdot \a D \phi\Rr].
\end{equation}
As was stated above, it is easy to compute the first term on the right side of~\eqref{e.basis}. This term is the linear part of the mapping going from coefficient fields to homogenized matrices. One way to understand why formula \eqref{e.basis} may be better suited than \eqref{e.charact.ahom} for the development of numerical methods is by considering the case of small ellipticity contrast. If we assume, say, that $1 \le \a \le 1 + \eta$ for some small constant $\eta > 0$, then the gradient $|D\phi|$ will be of the order of~$\eta$, and thus enters as a perturbation of the order of $\eta$ in the computation of $(\xi + D \phi) \cdot \a (\xi + D \phi)$. It is only after we take expectations that we can discover that the difference $|\ahom - \E[\a]|$ is in fact of the order of $\eta^2$. In contrast, this property is directly apparent in \eqref{e.basis}. The method presented in Theorem~\ref{t.main} can be understood as a way to leverage on this phenomenon in general, through a coarsening of the problem which progressively inherits the properties of the small-ellipticity-contrast regime.

\smallskip

Denoting by 
\begin{equation}
\label{e.def.f}
f(\a) := D \cdot \a \xi,
\end{equation}
and by $\la \cdot, \cdot \ra$ the scalar product in $L^2(\P)$, we can rewrite the second term on the right side of \eqref{e.basis} as
\begin{equation}  
\label{e.Dphi.flf}
\E\Ll[D\phi \cdot \a D \phi \Rr] = \lim_{\mu \to 0} \la f, (\mu + \L)^{-1} f \ra 
= \la f, \L^{-1} f \ra.
\end{equation}
Indeed, we have that $(\mu + \L)^{-1} f = \phi_\mu$ and 
\begin{equation*}  
\la f, \phi_\mu \ra = - \E \Ll[ D \phi_\mu \cdot \a \xi \Rr] \xrightarrow[\mu \to 0]{} -\E \Ll[ D \phi \cdot \a \xi \Rr] = \E \Ll[ D \phi \cdot \a D \phi \Rr] .
\end{equation*}
The problem of computing $\la f, \L^{-1} f \ra$ is precisely of the type discussed in the first part of the paper. Since the homogenization of the operator $-\nabla \cdot \a \nabla$ is essentially equivalent to the large-scale diffusive behavior of the corresponding random walk (which itself can be reduced to the study of an additive functional~\cite{kipvar}), this is not surprising. We point out that no reference to random walks is necessary to prove the results of Section~\ref{s.variants.gk} there, which are the only results of the first part we will need here. Indeed, all the arguments in Section~\ref{s.variants.gk} only rely on the classical resolvent formula \eqref{e.resolvent} and on the property~\eqref{e.hyp.int}. That the latter is true follows from \eqref{e.Dphi.flf} and the fact that, by reversibility, $\la f, e^{-t\L} f \ra \ge 0$. Hence, the reader who may have skipped the first part of the paper can read Section~\ref{s.variants.gk} independently of the rest of the first part, with the understanding that $\L$ and $f$ are defined by \eqref{e.def.mclL} and \eqref{e.def.f} respectively.

\smallskip

In order to compute $\la f, \L^{-1} f \ra$, we  employ the stragegy described in Theorem~\ref{t.pgk} with $\mu_k \equiv 2^{-k}$. That is, we set $f_{-1} = f$ and, for each $k \in \N$, denote by $f_k \in L^2(\Omega)$ the unique solution to
\begin{equation}  
\label{e.def.fk.homog}
(2^{-k} + \L) f_k = 2^{-k} f_{k-1}.
\end{equation}
We thus infer from Theorem~\ref{t.pgk} that, for each $n \in \N$,
\begin{equation}  
\label{e.series.homog}
\la f, \L^{-1} f \ra = \sum_{k = 0}^n 2^k \Ll( \la f_{k-1}, f_k \ra + \la f_k, f_k \ra \Rr) + \la f_n, \L^{-1} f_n \ra.
\end{equation}
For each $k \ge -1$, we denote by $(v_k(x), x \in \Z^d)$ the stationary extension of $f_k$, that is,
\begin{equation*}  
v_k(x) := f_k(\theta_x \, \a).
\end{equation*}
The dependence of $v_k$ on the coefficient field $\a$ is left implicit in the notation. Using \eqref{e.def.fk.homog}, one can check that $v_k$ solves \eqref{e.def.vk}.

\section{Error analysis}
\label{s.error}
The goal of this section is to prove the first part of Theorem~\ref{t.main}, which gives an estimate of the error in the approximation provided by $\hat{\sigma}^2_n$. Recall that $f$ is defined in \eqref{e.def.f} and the identities \eqref{e.basis} and \eqref{e.Dphi.flf}.
\begin{theorem}[Error estimate]
\label{t.eff} 
For each $\eps \in (0,\frac 1 2)$, there exists a constant $C < \infty$ such that the following holds. Fix $n \in \N$, and for each $k \le n$, denote
\begin{equation}  
\label{e.def.rk}
r_{k} := 2^{n-\Ll( \frac 1 2 - \eps \Rr) k}.
\end{equation}
We have
\begin{multline}
\label{e.eff}
\E\Ll[\Ll(\la f, \L^{-1} f \ra - \sum_{k = 0}^n \frac{1}{|B_{r_k}|} \sum_{x \in B_{r_k}} 2^k \Ll( v_{k-1}(x) v_k(x) + v_k^2(x) \Rr)\Rr)^2 \Rr]^\frac 1 2 \\\
\le  C 2^{-\frac{nd}{2}} . 
\end{multline}
\end{theorem}
\begin{remark}  
\label{e.sqrt.log}
The proof given below in fact gives the existence of a constant $C < \infty$ such that for any choice of $(r_k)_{0 \le k \le n}$, we have the inequality \eqref{e.eff} with the right side replaced by
\begin{equation*}  
C 2^{-\frac {dn}{2}} + C  \sum_{k = 0}^n 2^{-\frac {dk} 4} r_k^{-\frac d 2}. 
\end{equation*}
Fixing $r_k$ according to \eqref{e.def.rk} for some $\eps > 0$ indeed allows to bound the sum above by a constant times $2^{-\frac{dn}2}$, but other choices are possible, such as, say,
\begin{equation*}  
r_k = 2^{n - \frac k 2} \Ll(1+ \frac{k^\frac 3 d}{100}\Rr).
\end{equation*}
For $\eps = 0$, we obtain an upper bound of the form $C n 2^{-\frac{nd}{2}}$. I conjecture that the sharp upper bound in \eqref{e.eff} is
\begin{equation}
\label{e.cond.rk}
C 2^{-\frac {dn}{2}} + C 2^{-\frac{nd}{2}} \Ll(\sum_{k = 0}^n 2^{d \Ll(\frac k 2 - n\Rr)} r_k^d \Rr)^{\frac 1 2}, 
\end{equation}
which for the choice of $r_k$ as in \eqref{e.def.rk} with $\eps = 0$ gives the upper bound $C n^\frac 1 2 2^{-\frac{nd}{2}}$. 
As will be seen in the next section, for $r_k$ chosen as in \eqref{e.def.rk}, the asymptotic complexity of the method described in Theorem~\ref{t.main} does not depend on the choice of $\eps \in [0,\frac{d-1}{2d})$. 
\end{remark}

In order to prove Theorem~\ref{t.main}, our first task is to estimate the size of the remainder in the series expansion displayed in \eqref{e.series.homog}. 

\begin{proposition}[Remainder estimate]
\label{p.rest}
There exists a constant $C < \infty$ such that for every $n \in \N$,
\begin{equation*}  
0 \le \la f_n,  f_n \ra \le C 2^{-n \Ll( \frac d 2 + 1 \Rr) }
\end{equation*}
and 
\begin{equation*}  
0 \le \la f_n, \L^{-1} f_n \ra \le  C 2^{-\frac{nd}2}.
\end{equation*}
\end{proposition}
Proposition~\ref{p.rest} is an immediate consequence of Propositions~\ref{p.geom.step} and \ref{p.size.fk}, together with the second moment bound proved in \cite[Theorem~1]{GNO} and recalled in the next proposition.
\begin{proposition}[\cite{GNO}]
\label{p.moment.bound}
Recall that $f(\a) = D \cdot \a \xi$. For each $p \in [1,\infty)$, there exists a constant $C < \infty$ such that for every $t \ge 0$,
\begin{equation*}  
\E \Ll[ \Ll|e^{-t\L} f\Rr|^p \Rr] ^\frac 1 p \le C (1 + t)^{- \Ll( \frac d 4 + \frac 1 2 \Rr) }.
\end{equation*}
\end{proposition}
We rely on the spatial mixing property of the random field in \eqref{e.mixing.field} to compute the expected values in the series expansion in \eqref{e.series.homog}. This mixing property is quantified by the next proposition. For every random variable $X$, we denote by $\var[X] := \E[(X-\E[X])^2]$ the variance of $X$.

\begin{proposition}[Spatial mixing of $v_k$]
\label{p.spat.mix}
There exists $C(d,\Lambda) < \infty$ such that for every $k \in \N$ and $r \ge 0$,
\begin{equation}  
\label{e.spat.mix}
\var \Ll[ \frac{1}{|B_r|} \sum_{x \in B_r} \Ll(v_{k-1}(x) v_k(x) + v^2_k(x) \Rr)\Rr] \le C 2^{-2k \Ll( \frac d 2 + 1 \Rr) } \Ll( 1+2^{-\frac k 2} r  \Rr)^{-d}.
\end{equation}
\end{proposition}
As will be seen shortly, we have 
\begin{equation}  
\label{e.announce.fourth}
\E[v_k^4(x)] = \E[f_k^4] \le C 2^{-2k \Ll( \frac d 2 + 1 \Rr) },
\end{equation}
which is the same scaling as the square of the second moment, see Proposition~\ref{p.rest}. Roughly speaking, Proposition~\ref{p.spat.mix} shows that the correlation length of the random field $x \mapsto v_{k-1}(x) v_k(x) + v_k^2(x)$ is of the order of $2^{\frac k 2}$. Indeed, the fact that the left side of \eqref{e.spat.mix} is bounded by $C 2^{-2k \Ll( \frac d 2 + 1 \Rr)}$ is a direct consequence of the bound \eqref{e.announce.fourth}. The additional term on the right side of \eqref{e.spat.mix} witnesses CLT-type cancellations at length scale $2^{\frac k 2}$. 

\smallskip

In order to prove Proposition~\ref{p.spat.mix}, we define, for each $t \ge 0$ and $x \in \Z^d$,
\begin{equation*}  
u(t,x) := (e^{-t\L} f)(\theta_x \a).
\end{equation*}
The function $u$ solves the parabolic equation
\begin{equation}
\label{e.space.parab}
\Ll\{
\renewcommand{\arraystretch}{1.5}
\begin{array}{ll}  
\partial_t u   = \nabla \cdot \a \nabla u \quad & \text{ on } \R_+ \times \Zd, \\
u(0,\cdot) = \nabla \cdot \a \xi \quad & \text { on }  \Zd.
\end{array}
\Rr.
\end{equation}
We also denote by $(G(t,x,y), t \ge 0, x,y \in \Zd)$ the parabolic Green function associated with the operator $-\nabla \cdot \a\nabla$.
We will make use of \cite[Theorem~3]{GNO}, which we now recall. For each $e \in \B$, $x \in \Zd$ and $t \ge 0$, we introduce the weight function
\begin{equation*}  
w_e(t,x) := 1 + \frac{|x-e|}{(1+t)^\frac 1 2},
\end{equation*}
where we use the notation $|x-e|$ to denote the distance between $x$ and the endpoint of the edge $e$ closest to $x$. 
\begin{proposition}[\cite{GNO}]
\label{p.gradient.green}
There exists an exponent $p_0(d,\Lambda) > 1$ and, for every $\al < \infty$, a constant $C(d,\Lambda,\al) < \infty$ such that for every $1 \le p \le p_0$, $t \ge 0$ and $e \in \B$,
\begin{equation*}  
\Ll(\sum_{x \in \Zd} \Ll(w_e^\al(t,x) |\nabla G(t,x,e)|^2 \Rr)^p \Rr)^\frac 1 {2p} \le C (1 + t)^{- \Ll( \frac d 2 + \frac 1 2 \Rr) + \frac d 2 \frac 1 {2p}}.
\end{equation*}
\end{proposition}
We will also rely on the following observation.
\begin{lemma}
\label{l.gradient}
There exists $C(d,\Lambda) < \infty$ such that for every $t \ge 0$ and $e \in \mathbb{B}$,
\begin{equation*}  
\E\Ll[|\nabla u(t,e)|^2\Rr] \le C (1 + t)^{- \Ll( \frac d 2 + 2 \Rr)} .
\end{equation*}
\end{lemma}
\begin{proof}
We first observe that the mapping $t \mapsto \E[(\nabla u \cdot \a \nabla u)(t,x)]$ is nonincreasing. Indeed, by stationarity, and dropping the variables $(t,x)$ from the notation,
\begin{align*}  
\partial_t \E[\nabla u \cdot \a \nabla u] & = 2 \E \Ll[ \nabla  \Ll( \nabla \cdot \a \nabla u \Rr) \cdot \a \nabla u \Rr] \\
& = - 2 \E \Ll[ \Ll( \nabla \cdot \a \nabla u \Rr) ^2 \Rr] \le 0.
\end{align*}
Moreover,
\begin{equation*}  
\partial_t \E[u^2] = -2  \E \Ll[ \nabla u \cdot \a \nabla u \Rr].
\end{equation*}
Using also Proposition~\ref{p.moment.bound} with $p = 2$, we obtain
\begin{align*}  
\E[(\nabla u \cdot \a \nabla u)(t,x)] & \le \frac 2 t \int_{\frac t 2}^t \E[(\nabla u \cdot \a \nabla u)(s,x)] \, \d s \\
& = t^{-1} \Ll( \E \Ll[ u^2\Ll(\tfrac t 2, x\Rr) \Rr] - \E\Ll[u^2(t,x)\Rr] \Rr) \\
& \le C (1+t)^{-\Ll( \frac d 2 + 2 \Rr) }.
\end{align*}
The announced result then follows by the assumption of uniform ellipticity of the coefficient field. 
\end{proof}

\begin{proof}[Proof of Proposition~\ref{p.spat.mix}]
For notational convenience, we focus on showing that 
\begin{equation}  
\label{e.spat.mix2}
\var \Ll[ \frac 1 {|B_r|} \sum_{x \in B_r} v_k^2(x) \Rr] \le C 2^{-2k \Ll( \frac d 2 + 1 \Rr) } \Ll( 1 + 2^{-\frac k 2} r \Rr) ^{-d}.
\end{equation}
The proof that
\begin{equation*}  
\var \Ll[ \frac 1 {|B_r|} \sum_{x \in B_r} v_{k-1}(x) v_k(x) \Rr] \le C 2^{-2k \Ll( \frac d 2 + 1 \Rr) } \Ll( 1 + 2^{-\frac k 2} r \Rr) ^{-d}
\end{equation*}
is essentially identical, and we therefore omit it. We decompose the proof into four steps. 

\smallskip

\emph{Step 1.}
We show that there exists a constant $C < \infty$ such that for every $k \in \N$ and $x \in \Zd$,
\begin{equation}  
\label{e.fourth.moment}
\E \Ll[ v_k^4(x) \Rr] \le C 2^{-k \Ll( d + 2 \Rr) }.
\end{equation}
Denote by $S_k$ the random variable in \eqref{e.def.Sn}. We assume that this random variable is independent of the coefficient field $\a$, and denote the expectation with respect to this random variable by $E$. We recall that we have the representation
\begin{equation}  
\label{e.recall.rep}
v_k(x) = E \Ll[ u(S_k,x) \Rr] .
\end{equation}
By Jensen's inequality and Proposition~\ref{p.moment.bound} (with $p = 4$), we have
\begin{equation*}  
\E[v_k^4(x)] \le E \E  \Ll[ u^4(S_k,x) \Rr] \le E \Ll[ (1+S_k)^{-(d+2)} \Rr] .
\end{equation*}
The conclusion \eqref{e.fourth.moment} then follows from \eqref{e.Sn.moment}.

\smallskip

\emph{Step 2.}
By Jensen's inequality and \eqref{e.fourth.moment}, we obtain the inequality \eqref{e.spat.mix2} provided that $r \le 2^{\frac k 2}$. From now on, we therefore assume that $r \ge 2^{\frac k 2}$. We start by recalling the Efron-Stein (or spectral gap) inequality. Let $(\a'_e)_{e \in \B}$ be an independent copy of the environment $(\a_e)_{e \in \B}$, defined on the same probability space. For each edge $e \in \B$, we let $\a^e$ be the environment defined by
\begin{equation}
\label{e.def.ae}
\a^e_b := 
\Ll|
\begin{array}{ll}
\a_b & \text{if } b \neq e, \\
\a'_e & \text{if } b = e.
\end{array}
\Rr.
\end{equation}
In other words, the environment $\a^e$ is obtained from the environment $\a$ by resampling the conductance at the edge $e$. By the independence assumption on the conductances~$(\a_e)$, every random variable $X$ satisfies the Efron-Stein inequality
\begin{align}\label{SG}
\var[X] \le \frac 1 2 \sum_{e \in \B} \E \Ll[  \Ll(\partial_e X \Rr)^2 \Rr],
\end{align}
where $\partial_e X$ denotes the Glauber derivative:
$$
\partial_eX(a):= X(\a^e)-X(\a).
$$
Applying this inequality to our problem yields that
\begin{equation*}  
\var \Ll[ \sum_{x \in B_r} v_k^2(x) \Rr] \le \frac 1 2 \sum_{e \in \B} \E \Ll[ \Ll( \sum_{x \in B_r} \partial_{e} \Ll( v_k^2(x)\Rr) \Rr) ^2 \Rr] .
\end{equation*}
For each edge $e$, we write
\begin{align*}  
v_{k}^{(e)}(x) & := |v_k(\a^e,x)| + |v_k(\a,x)|.
\end{align*}
where we stressed the dependency on the environment $\a$ or $\a^e$ in the notation. Note that the environment $\a^e$ has the same law as $\a$. Moreover,
\begin{equation*}  
\Ll|\partial_e \Ll( v_k^2(x) \Rr) \Rr|\le v_{k}^{(e)}(x) \, \Ll|\partial_e v_k(x)\Rr|.
\end{equation*}
We fix an exponent $\al > d$. By H\"older's inequality,
\begin{align*}  
\sum_{x \in B_r} \partial_{e} \Ll( v_k^2(x)\Rr) & \le \sum_{x \in B_r} v_k^{(e)}(x) \Ll| \partial_e v_k(x) \Rr| \\
& \le \Ll( \sum_{x \in B_r} w_e^{-\al}(2^k,x) |v_k^{(e)}(x)|^2 \Rr)^\frac 1 2 \Ll( \sum_{x \in B_r} w_e^\al(2^k,x) |\partial_e v_k(x)|^2 \Rr)^\frac 1 2.
\end{align*} 
Therefore,
\begin{multline}  
\label{e.full.var}
\var \Ll[ \sum_{x \in B_r} v_k^2(x) \Rr] \\
 \le \sum_{e \in \B}\E \Ll[ \Ll(\sum_{x \in B_r} w_e^{-\al}(2^k,x) |v_k^{(e)}(x)|^2 \Rr)^2\Rr]^\frac 1 2 \E \Ll[ \Ll(\sum_{x \in B_r} w_e^\al(2^k,x) |\partial_e v_k(x)|^2 \Rr)^2\Rr]^\frac 1 2.
\end{multline}
By H\"older's inequality,
\begin{equation*}  
\Ll(\sum_{x \in B_r} w_e^{-\al}(2^k,x) |v_k^{(e)}(x)|^2 \Rr)^2\le \Ll(\sum_{x \in B_r} w_e^{-\al}(2^k,x)\Rr)  \Ll( \sum_{x \in B_r} w_e^{-\al}(2^k,x) |v_k^{(e)}(x)|^4 \Rr).
\end{equation*}
Hence, by the definition of $v_k^{(e)}$ and \eqref{e.fourth.moment}, 
\begin{equation}  
\label{e.bound2}
\E \Ll[ \Ll(\sum_{x \in B_r} w_e^{-\al}(2^k,x) |v_k^{(e)}(x)|^2 \Rr)^2\Rr]^\frac 1 2\le C 2^{-k} \Ll( 1 + 2^{-\frac k 2} \dist(e,B_r) \Rr) ^{-\al}.
\end{equation}
If we temporarily admit that there exists a constant $C(d,\Lambda,\alpha) < \infty$ such that for every $e \in \B$,
\begin{equation}
\label{e.bound1}
\E \Ll[ \Ll( \sum_{x \in B_r} w_e^\al(2^k,x) |\partial_e v_k(x)|^2 \Rr) ^2 \Rr]^\frac 1 2 \le C 2^{-k \Ll( \frac d 2 + 1 \Rr) },
\end{equation}
then we can combine the two bounds \eqref{e.bound2} and \eqref{e.bound1} into \eqref{e.full.var} and thus complete the proof of the inequality \eqref{e.spat.mix2}. It therefore only remains to show that \eqref{e.bound1} holds. We decompose the proof of this latter statement into two steps. 

\smallskip

\emph{Step 3.} We show that there exists a constant $C < \infty$ such that for every $r,t \ge 0$ and $e \in \B$, 
\begin{equation}
\label{e.bound1.ut}
\E \Ll[ \Ll( \sum_{x \in B_r} w_e^\al(t,x) |\partial_e u(t,x)|^2 \Rr) ^2 \Rr]^\frac 1 2 \le C (1 + t)^{- \Ll( \frac d 2 + 1 \Rr) }.
\end{equation}
We recall from \cite[Lemma~3]{GNO} that there exists $C(d,\Lambda) < \infty$ such that
\begin{equation}  
\label{e.decomp.deu}
|\partial_e u(t,x)| \le C\Ll|\nabla G(t,x,e)\Rr| + \int_0^t \Ll| \nabla G(t-s,x,e) g(s,e) \Rr| \d s ,
\end{equation}
where $g(s,e)$ is a random variable that satisfies, for each $p < \infty$,
\begin{equation}
\label{e.boundg}
\E \Ll[ |g(s,e)|^p \Rr]^\frac 1 p \le C \E \Ll[ |\nabla u(s,e)|^p \Rr]^\frac 1 p .
\end{equation}
By Proposition~\ref{p.gradient.green} with $p = 1$, the first term on the right side of \eqref{e.decomp.deu} is handled without difficulty. 
It therefore suffices to show that
\begin{multline}  
\label{e.bound11}
\E \Ll[ \Ll( \sum_{x \in \Z^d} w_e^\al(t,x) \Ll| \int_0^t \Ll| \nabla G(t-s,x,e) g(s,e) \Rr| \, \d s  \Rr|^2  \Rr)^2  \Rr]^\frac 1 2 \\
\le C (1 + t)^{- \Ll( \frac d 2 + 1 \Rr) }.
\end{multline}
We first note for future reference that by H\"older's inequality, \eqref{e.boundg}, Lemma~\ref{l.gradient} and Proposition~\ref{p.moment.bound}, there exists a constant $C(d,\Lambda) < \infty$ such that
\begin{equation}  
\label{e.boundg2}
\E \Ll[ |g(s,e)|^4 \Rr]^\frac 1 4 \le \E[|g(s,e)|^2]^\frac 1 2 \, \E[|g(s,e)|^6]^\frac 1 6 \le C t^{- \Ll( \frac d 4 + 1 \Rr) }.
\end{equation}
We now proceed to estimate the sum on the left side of \eqref{e.bound11}.
By the integral triangle inequality,
\begin{multline*}  
\Ll(\sum_{x \in \Z^d} w_e^\al(t,x) \Ll| \int_0^t \Ll| \nabla G(t-s,x,e) g(s,e) \Rr| \, \d s  \Rr|^2\Rr)^\frac 1 2 \\
\le \int_0^t  \Ll( \sum_{x \in \Z^d} w_e^\al(t,x) \Ll| \nabla G(t-s,x,e) g(s,e) \Rr|^2 \Rr)^\frac 1 2 \, \d s .
\end{multline*}
Applying the triangle inequality again, we arrive at
\begin{multline}  \label{e.bound.11.int}
\E \Ll[ \Ll( \sum_{x \in \Z^d} w_e^\al(t,x) \Ll| \int_0^t \Ll| \nabla G(t-s,x,e) g(s,e) \Rr| \, \d s  \Rr|^2  \Rr)^2  \Rr]^\frac 1 {4} \\
\le \int_0^t \E \Ll[ \Ll(  \sum_{x \in \Z^d} w_e^\al(t,x) \Ll| \nabla G(t-s,x,e) g(s,e) \Rr|^2\Rr)^{2}  \Rr]^{\frac 1 {4}} \, \d s.
\end{multline}
We fix $p = p_0(d,\Lambda) > 1$ as given by Proposition~\ref{p.gradient.green}, and denote by $q =\frac{p}{p-1}$ its conjugate exponent. We also give ourselves an exponent $\be > 0$ to be specified later. By H\"older's inequality,
\begin{align*}  
& \sum_{x \in \Z^d} w_e^\al(t,x) \Ll| \nabla G(t-s,x,e) g(s,e) \Rr|^2 
\\&  \quad \le \Ll(\sum_{x \in \Z^d} w_e^\al(t,x) w_e^{\be p}(t-s,x) \Ll| \nabla G(t-s,x,e) \Rr|^{2p}  \Rr)^\frac 1 {p}
 \\ & \qquad \times  \Ll(\sum_{x \in \Z^d} w_e^\al(t,x) w_e^{-\be q}(t-s,x) \Ll| g(s,e) \Rr|^{2q}  \Rr)^\frac 1 {q}
\end{align*}
By Proposition~\ref{p.gradient.green}, the first term is deterministically bounded by a constant times
\begin{equation*}  
(1 + |t-s|)^{- \Ll( d + 1 \Rr)  + \frac d 2 \frac {1}{p}}.
\end{equation*}
Moreover, using the triangle inequality once more, we get
\begin{align*}  
& \E \Ll[ \Ll(\sum_{x \in \Z^d} w_e^\al(t,x) w_e^{-\be q}(t-s,x) \Ll| g(s,e) \Rr|^{2q}  \Rr)^\frac 2 {q} \Rr]^\frac {q} 2 \\
& \quad \le \sum_{x \in \Z^d} \E \Ll[ \Ll( w_e^\al(t,x) w_e^{-\be q}(t-s,x) \Ll| g(s,e) \Rr|^{2q}  \Rr)^\frac 2 {q}  \Rr]^\frac{q} 2 \\
& \quad \le C \Ll(1 + |t-s|\Rr)^{\frac d 2} \, \Ll(  1 + s \Rr)^{-2q \Ll( \frac d 4 + 1 \Rr)} ,
\end{align*}
where we chose $\be < \infty$ sufficiently large and used \eqref{e.boundg2} and Proposition~\ref{p.moment.bound} in the last step. 
The quantity on the left of \eqref{e.bound.11.int} is therefore bounded by a constant times
\begin{align*}  
& \int_0^t (1 + |t-s|)^{- \Ll( \frac d 2  + \frac 1 2 \Rr)  + \frac d 2 \frac {1}{2p}} \, (1 + |t-s|)^{\frac d 2 \frac 1 {2q}} \, \Ll( 1 + s \Rr)^{-\Ll( \frac d 4 + 1 \Rr)}  \, \d s \\
& \qquad = \int_0^t (1 + |t-s|)^{- \Ll( \frac d 4 + \frac 1 2 \Rr) } \, (1 + s)^{- \Ll( \frac d 4 + 1 \Rr)} \, \d s \\
& \qquad \le C (1 + t)^{-\Ll(\frac d 4 + \frac 1 2\Rr)}.
\end{align*}
This completes the proof of the inequality \eqref{e.bound11}.

\smallskip

\emph{Step 4.} In order to complete the proof, there remains to show that \eqref{e.bound1.ut} implies \eqref{e.bound1}. We use the representation~\eqref{e.recall.rep} and Jensen's and the triangle inequalities to bound
\begin{align*}  
& \E \Ll[ \Ll( \sum_{x \in B_r} w_e^\al(2^k,x) |\partial_e v_k(x)|^2 \Rr) ^2 \Rr]^\frac 1 2 
\\
& \qquad  \le \E \Ll[ \Ll(E\Ll[\sum_{x \in B_r} w_e^\al(2^k,x) |\partial_e u(S_k,x)|^2 \Rr] \Rr)^2 \Rr]^\frac 1 2 \\
& \qquad \le E \Ll[ \E \Ll[ \Ll(\sum_{x \in B_r} w_e^\al(2^k,x) |\partial_e u(S_k,x)|^2 \Rr)^2 \Rr]^\frac 1 2  \Rr] .
\end{align*}
By the result of the previous step and \eqref{e.Sn.moment}, this last term is bounded by a constant times 
\begin{equation*}  
 E \Ll[ (1+S_k)^{-\Ll( \frac d 2 + 1 \Rr) } \Rr] \le C 2^{-k \Ll( \frac d 2 + 1 \Rr) }. 
\end{equation*}
This completes the proof.
\end{proof}
\begin{proof}[Proof of Theorem~\ref{t.eff}]
We decompose the left side of \eqref{e.eff} into bias and variance. Using Theorem~\ref{t.pgk} and Proposition~\ref{p.rest}, we can bound the bias by
\begin{align*}  
 \Ll| \la f, \L^{-1} f  \ra - \sum_{k = 0}^{n}  \frac{1}{|B_{r_k}|} \sum_{x \in B_{r_k}} 2^k \E \Ll[ v_{k-1}(x) v_k(x) + v_k^2(x) \Rr]  \Rr| & = \Ll| \la f_n, \L^{-1} f_n \ra \Rr| \\
 & \le C 2^{-\frac {dn}{2}}.
\end{align*}
By Proposition~\ref{p.spat.mix}, we have
\begin{align*}  
\var \Ll[ \frac{1}{|B_{r_k}|} \sum_{x \in B_{r_k}} 2^k \Ll( v_{k-1}(x) v_k(x) + v_k^2(x) \Rr) \Rr] 
& \le C 2^{-\frac{dk}{2}} r_k^{-d} \\
& \le C 2^{-dn-\eps d k}.
\end{align*}
By the triangle inequality, we deduce that
\begin{align*}  
\var \Ll[ \sum_{k = 0}^n \frac{1}{|B_{r_k}|} \sum_{x \in B_{r_k}} 2^k \Ll( v_{k-1}(x) v_k(x) + v_k^2(x) \Rr)\Rr] 
& \le C \Ll( \sum_{k = 0}^{n} 2^{-\frac 1 2 \Ll(dn + \eps d k\Rr)}  \Rr)^2 \\
& \le C 2^{-dn}.
\end{align*}
This completes the proof.
\end{proof}

\section{Complexity analysis}
\label{s.compl}

In this section, we first complete the proof of Theorem~\ref{t.main} by evaluating the computational complexity of the method described there. We then prove the complexity lower bound announced in Proposition~\ref{p.lower.bound}. 

\smallskip

We recall that the method described in Theorem~\ref{t.main} is a family of approximations of the quantity of interest given, for each $n \in \N$, by
\begin{equation}  
\label{e.this.is.the.method}
\sum_{k = 0}^n \frac{1}{|B_{r_k}|} \sum_{x \in B_{r_k}} 2^k \Ll( v_{k-1}(x) v_k(x) + v_k^2(x) \Rr),
\end{equation}
where $v_k$ are functions on $\Z^d$ defined by \eqref{e.def.vk}, and $r_k$ is defined in \eqref{e.def.main.rk} and depends implicitly on $n$ and on an exponent $\eps \in [0,\frac 1 2)$. (In view of Remark~\ref{e.sqrt.log}, we will include the case $\eps = 0$ in the complexity analysis.) We investigate the complexity of this method under the additional assumption
\begin{equation}  
\label{e.assumption.epsilon}
\eps < \frac{d-1}{2d}. 
\end{equation}
The question concerns the cost of computing the functions $(v_k, k \in \{0,\ldots n\})$ on $B_{r_k}$. These functions are solutions of elliptic equations. As we now argue, direct conjugate gradient methods, without any preconditioning, already yield the optimal scaling for the complexity of the method. Indeed, by standard Green function upper bounds (see e.g.\ \cite[Proposition~3.6]{dl-diff}), if we solve the equation for $v_k$ with Dirichlet boundary condition outside of $B_{r_k + C n 2^{\frac k 2}}$ for a sufficiently large constant $C < \infty$ (instead of the full space), then we can ensure that the solution thus obtained is $2^{-100dn}$ away from the true solution $v_k$. The problem of computing \eqref{e.this.is.the.method} therefore boils down to the resolution of a series of elliptic problems, indexed by $k \in \{0,\ldots, n\}$, where for each $k$, we need to invert the operator $(2^{-k} + \nabla \cdot \a \nabla)$ on a domain of side length $r_k + C n 2^{\frac k 2}$. The condition number for this problem is of the order of $2^k$. Recall that using the conjugate gradient method, the computational cost of computing the solution to an elliptic problem with $N$ unknowns and condition number $\kappa$ at precision $\delta$ is of the order of $N \sqrt{\kappa} \log(\delta^{-1})$ (see e.g.\ \cite[Theorem~4.12]{quart}). Hence, for the problem of computing \eqref{e.this.is.the.method}, the total complexity using conjugate gradient methods without preconditioning, and with a precision fixed at, say, $2^{-100dn}$, is bounded by a constant times
\begin{align*}  
n\sum_{k = 0}^n \Ll(r_k + C n 2^{\frac k 2}\Rr)^d 2^{\frac k 2} & = n\sum_{k = 0}^n \Ll(2^{n-\frac k 2 + \eps k} + C n 2^{\frac k 2}\Rr)^d 2^{\frac k 2} \\
& \le C n\sum_{k = 0}^n 2^{nd - \Ll( \frac {d-1} 2 - \eps d \Rr) k} \\
& \le C n2^{nd},
\end{align*}
where we used the assumption of \eqref{e.assumption.epsilon} in the last step. This completes the proof of Theorem~\ref{t.main}.

\smallskip

One may argue that, under the specific i.i.d.\ assumption we make on the coefficient field, periodization would allow to simply get rid of any boundary layer. However, the idea of periodizing the medium is in general not suitable for solving the problem as stated at the onset of the introduction to this paper. Indeed, as soon as some short range of correlation is present in the medium, a periodization of the medium will create boundary defects that cause significant errors. This is discussed in more details in \cite{cemracs}, see in particular Figure~9 there.

\smallskip

In the method described in Theorem~\ref{t.main}, the computational overhead caused by the necessity of boundary layers is an asymptotically small fraction of the total computational cost. More precisely, if we compare, say for $\eps = 0$, the actual cost of the method with the fictitious cost one would face if boundary layers were not computed (i.e.\ where only values that enter into the formula \eqref{e.this.is.the.method} are computed), we find that the difference is bounded by a constant times
\begin{equation*}  
n^2 \sum_{k = 0}^n 2^{\frac k 2 + (d-1) \Ll( n - \frac k 2 \Rr)  + \frac k 2} \le C n^2 2^{(d-1)n} \times
\Ll|
\begin{array}{ll}  
2^{\frac n 2} & \quad \text{if } d = 2, \\
n & \quad \text{if } d = 3, \\
1  &\quad \text{if } d \ge 4.
\end{array}
\Rr.
\end{equation*}
Hence, up to logarithmic corrections, the computational overhead of boundary layers is only of the order of $2^{(d-1)n}$ in dimension $d \ge 3$. As we will argue shortly, no method can compute the homogenized coefficients at precision $2^{-\frac{nd}{2}}$ without evaluating the coefficients in a region of space of volume $2^{nd}$, and in this case a boundary layer of unit size already gives a computational overhead of $2^{(d-1)n}$, so this estimate is optimal.

\smallskip

We now turn to proving the complexity lower bound announced in Proposition~\ref{p.lower.bound}. Since this proposition was stated relatively loosely, we first make it more precise. The number of operations of a given algorithm is a random variable; at precision $\de > 0$, we call this random variable $\mcl N_\de$. The assumption we wish to contradict is that $\de^2 \mcl N_\de$ tends to $0$ in probability, as $\de$ tends to $0$. We will contradict it for a specific choice of the law of the random environment. However, it will be clear that the proof is very generic, and that the specific choice we make is only a matter of convenient notation and computation. Indeed, the only key point is to observe that ``generically'', the Kullback-Leibler divergence (a.k.a.\ the relative entropy) responds quadratically to perturbations.

\begin{proof}[Proof of Proposition~\ref{p.lower.bound}]
For each $p \in [0,1]$, let $\P_p$ denote the product measure
\begin{equation*}  
\P_p = \bigotimes_{e \in \B}\,  (p \de_1 + (1-p) \de_2),
\end{equation*}
where $\de_x$ denotes the Dirac probability measure at $x \in \R$. We denote by $\E_p$ the associated expectation, and take $(\a_e)_{e \in \B}$ to be the canonical random variable on $\R^\B$. Let $\ahom(p)$ be the homogenized matrix associated with the law $\P_p$. By \cite{dl-diff}, the mapping $p \mapsto \ahom(p)$ is $C^1$ on $[0,1]$, and it is not constant since $\ahom(0) \neq \ahom(1)$. We will argue that in order for $\ahom(p)$ to be determined within an error of $\de$, the algorithm must query the values of at least $\simeq \de^{-2}$ conductances, and therefore perform at least $\simeq \de^{-2}$ operations. Since the law is of product form, we may assume without loss of generality that the algorithm queries $\a_{e_1}, \a_{e_2},\ldots$ in order, where $e_1,e_2,\ldots$ is an arbitrary sequence of distinct elements of $\B$. 

For every $p,q \in (0,1)$, every $n \in \N$ and every event $A \in \sigma \Ll( \a_{e_1},\ldots,\a_{e_n}, e \in \mcl E \Rr) $,
\begin{align}  
\P_q \Ll[ A \Rr] & = \E_p \Ll[ \1_A \prod_{i = 1}^n \Ll(\frac q p \1_{\a_{e_i} = 1} + \frac{1-q}{1-p} \1_{\a_{e_i} = 2} \Rr)\Rr] \notag\\
& = \E_p \Ll[ \1_A \exp \Ll( \hat{\mathrm{kl}}(q,p,n) \Rr)  \Rr] ,
\label{e.hatkl}
\end{align}
where we introduced the ``empirical'' Kullback-Leibler divergence
\begin{equation*}  
\hat{\mathrm{kl}}(q,p,n) := \sum_{i = 1}^n \log\Ll(\frac q p \1_{\a_{e_i} = 1} + \frac{1-q}{1-p} \1_{\a_{e_i} = 2} \Rr).
\end{equation*}
The ``true'' Kullback-Leibler divergence between Bernoulli random variables with parameters $q$ and $p$ is the expectation of each of these summands, that is,
\begin{equation*}  
\mathrm{kl}(q,p) := p \log \frac q p + (1-p) \log \frac{1-q}{1-p} \ge 0.
\end{equation*}
Differentiating with respect to $q$, we verify that for each $\eta \in (0,1/2)$, there exists $C(\eta) < \infty$ such that 
\begin{equation*}  
\forall p,q \in (\eta,1-\eta), \quad \mathrm{kl}(q,p) \le C(q-p)^2.
\end{equation*}
We also verify that 
\begin{equation*}  
\forall p,q \in (\eta,1-\eta), \quad \log\Ll(\frac q p \1_{\a_{e_i} = 1} + \frac{1-q}{1-p} \1_{\a_{e_i} = 2} \Rr) \le C|q-p|.
\end{equation*}
In particular, uniformly over $p,q \in (\eta,1-\eta)$ and $n \ge 1$,
\begin{equation}  
\label{e.control.hatkl}
\E_p \Ll[ \Ll(\hat{\mathrm{kl}}(q,p,n) \Rr)^2\Rr] \le C n |q-p|^2 + C n^2 |q-p|^4.
\end{equation}

If an algorithm can compute $\had$ in $o(\de^{-2})$ operations, we can define a deterministic $N_\de = o(\de^{-2})$ such that the algorithm queries only the  edges $e_1,\ldots,e_{N_\de}$ with probability at least $1/2$. Denoting by $\mcl E_\de$ the event that only edges $e_1,\ldots,e_{N_\de}$ are queried, we have
\begin{align*}  
\E_{q} \Ll[ (\had - \ahom(q))^2 \Rr] & \ge \E_{p} \Ll[ (\had - \ahom(q))^2   \, \1_{\mcl E_\de}  \, \exp \Ll( \hat{\mathrm{kl}}(q,p,N_\de)\Rr) \Rr].
\end{align*}
We choose $q_\de$ in such a way that
\begin{equation}
\label{e.def.qde}
\de \ll q_\de-p \ll N_\de^{-1/2},
\end{equation}
in the sense that the ratio of the right to the left side of both inequalities diverge to infinity as $\de \to 0$. Let $\mcl E'_\de$ be the conjunction of $\mcl E_\de$ with the event that $\hat{\mathrm{kl}}(q_\de,p,N_\de) \ge -1$. 
Combining the three previous displays, we obtain, for $\de > 0$ sufficiently small,
\begin{equation*}  
\E_{q_\de} \Ll[ |\had - \ahom(q_\de)|^2 \Rr] \ge \E_p\Ll[ |\had - \ahom(q_\de)|^2   \, \1_{\mcl E'_\de}   \Rr] \, \exp(-1) , \qquad \P_p \Ll[ \mcl E'_\de \Rr] \ge 1/4.
\end{equation*}
Since $\E_p \Ll[ \Ll| \had - \ahom(p) \Rr| ^2 \Rr] \le \de^2$, we deduce that
\begin{align*}  
|\ahom(p) - \ahom(q_\de)| & \le 4 \E_p[|\ahom(p) - \ahom(q_\de)| \1_{\mcl E'_\de}] \\
& \le 4 \E_p \Ll[ |\had - \ahom(q_\de)|^2 \1_{\mcl E'_\de}\Rr]^{1/2} + 4 \E_p \Ll[ |\had - \ahom(p)|^2 \1_{\mcl E'_\de}\Rr]^{1/2} \\
& \le 8 \E_{q_\de} \Ll[ |\had - \ahom(q_\de)|^2 \Rr]^{1/2}  + 4\de \\
& \le 12 \de.
\end{align*}
Since our choice of $p \in (0,1)$ was arbitrary, and $p \mapsto \ahom(p)$ is $C^1$ and non-constant over the interval $[0,1]$, this contradicts \eqref{e.def.qde}.
\end{proof}

\section{Numerical tests}
\label{s.numer}

In this section, we report on numerical results for the method presented in Theorem~\ref{t.main}. For comparison, we also discuss the performance of the previously used method. The code was written in the Julia language, and can be downloaded as part of the source files of the arXiv posting of this paper. 

\smallskip

We focus on a two-dimensional example, with conductances taking the values $c_- = 1$ and $c_+ = 9$ with probability $1/2$ each. This is a rare example for which the homogenized matrix is know analytically: it is $\ahom = \sqrt{c_- c_+} \, \mathrm{Id} = 3 \, \mathrm{Id}$ (see e.g.\ \cite[Appendix~A]{Gloria} or \cite[Exercise~2.3]{AKMBook}). In this case, we have $\E[\a] = 5 \, \mathrm{Id}$, and we thus expect that $\hat{\sigma}_n^2$ converges to $2$. We fix $\xi = (1,0)$.

\smallskip

The best numerical method should probably optimize the parameter $r_k$ in a dynamic way rather than rely on an a length scale obtained a priori from theoretical arguments. Here, I strived instead to avoid any fine-tuning and simply chose $\eps = 0$ in the method described in Theorem~\ref{t.eff}. Whenever the operator $(\mu - \nabla \cdot \a \nabla)$ needs to be inverted, I chose the size of the boundary layer to be
\begin{equation}  
\label{e.boundary.size}
5 \Ll(1 \vee \mu^{-\frac 1 2}\Rr) \Ll(1 \vee \frac{\log \Ll(\mu^{-\frac 1 2}\Rr)}{\log(2)}\Rr) ,
\end{equation}
where we recall that $x \vee y := \max(x,y)$. 

\smallskip

The base-2 logarithm of the left side of \eqref{e.main} (with $\eps = 0$), as a function of $n$, is represented on Figure~\ref{f.error.itresolv}. The slope of the regression line there is $-1.0$, in full agreement with the theoretical prediction of Theorem~\ref{t.main}. Moreover, the intercept of the regression line is $1.0$, which indicates that the prefactor constant in \eqref{e.main} can be chosen as $C \simeq 2.0$. 

\smallskip

\begin{figure}
\centering
\includegraphics[scale=0.5,trim=0.5cm 2cm 0cm 1cm, clip=true]{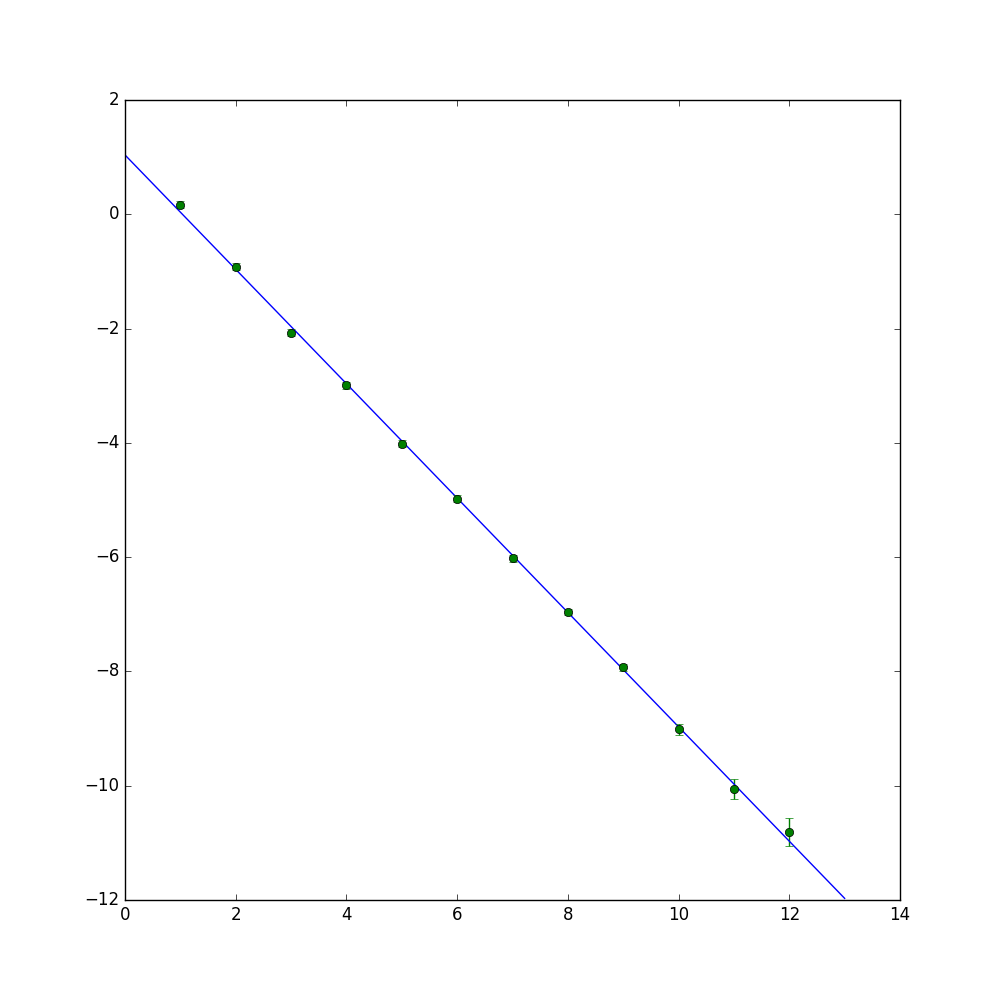}
\caption{
\small{
The base-2 logarithm of the left side of \eqref{e.main} (with $\eps = 0$) as a function of $n$. The error bars are $95\%$ confidence intervals. The slope of the regression line is $-1.0$ and the intercept is $1.0$.
}
}
\label{f.error.itresolv}
\end{figure}

The code was run on a laptop computer with $16$ Go of memory and a processor clocking at $2.40$ GHz. The base-2 logarithm of the time it takes to compute $\hat{\sigma}^2_n$ on this machine, as a function of $n$, is reported on Figure~\ref{f.cpu.itresolv}. 
The slope of the regression line there is $1.9$, very closely agreeing with the prediction of Theorem~\ref{t.main}. The intercept is $-15.2$. It takes less than $3$ minutes to compute~$\hat{\sigma}^2_{12}$.

\begin{figure}
\centering
\includegraphics[scale=0.5,trim=0.5cm 2cm 0cm 1cm, clip=true]{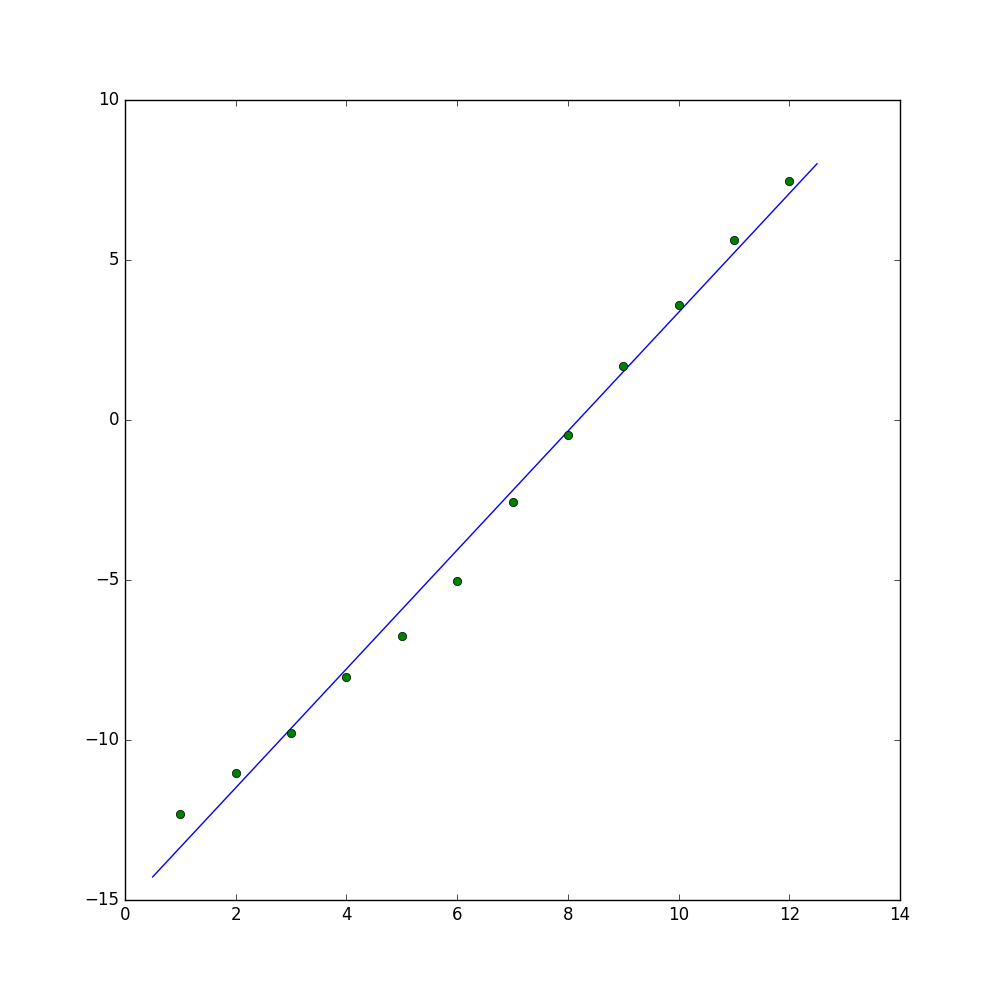}
\caption{
\small{
The base-2 logarithm of the CPU time used to compute $\hat{\sigma}^2_n$, in seconds, as a function of $n$. The slope of the regression line is $1.9$ and the intercept is $-15.2$. 
}
}
\label{f.cpu.itresolv}
\end{figure}


\smallskip

We compare these results with the following more classical approach. Recall the definition of the approximate corrector $\phi_\mu$ in \eqref{e.def.phimu}. For each $n \in \N$, we denote $\mu_n := 2^{-n}$ and compute $2^n$ independent copies $\phi_{\mu_n}^{(1)}, \ldots, \phi_{\mu_n}^{(2^n)}$ of $\phi_{\mu_n}$ on a (two-dimensional) box of side length $2^\frac n 2$. (This involves the management of boundary layers of side length $C n 2^\frac n 2$.) The algorithm then outputs
\begin{equation*}  
\tilde \sigma_n^2 := \frac 1 {k_n} \sum_{i = 1}^{k_n} \frac 1 {|B_{2^{\frac n 2}}|} \sum_{x \in B_{2^{\frac n 2}}} (\xi + \nabla \phi_{\mu_n}^{(i)}) \cdot \a (\xi + \nabla \phi_{\mu_n}^{(i)})) (x).
\end{equation*}
It was shown in \cite{GO1, GO2} that there exists a constant $C < \infty$ such that, for every $n \in \N$,
\begin{equation}  
\label{e.l2.classical}
\E \Ll[ \Ll(\xi \cdot \ahom \xi - \tilde \sigma_n^2 \Rr)^2 \Rr] ^\frac 1 2 \le C 2^{-n}.
\end{equation}
(Recall that we assume $d = 2$ here.)  We stick with the choice of \eqref{e.boundary.size} for the size of the boundary layer. 

\smallskip

Figure~\ref{f.error.classical} displays the base-2 logarithm of the left side of \eqref{e.l2.classical}, as a function of $n$. The slope of the regression line is $-0.9$, in very good agreement with the theoretical prediction. Figure~\ref{f.cpu.classical} reports on the cost of computing $\tilde \sigma_n^2$, as a function of $n$. Two implementations were tried. The first uses no preconditioning, and in this case the slope of the regression line in Figure~\ref{f.cpu.classical} is about $2.7$, in close agreement with the theoretical prediction of $\frac 5 2$. I also implemented a version which uses an incomplete Cholesky factorization as a preconditioner. While this does reduce the number of iterations of the conjugate gradient method, this gain was not sufficient to offset the overhead caused by the preconditioning. The slope of the regression line on Figure~\ref{f.cpu.classical} in this case is $2.6$, but the intercept is $-12.1$, while the regression line for the non-preconditioned method has an intercept of $-15.0$. It takes more than $2$ minutes to compute $\tilde \sigma_8^2$ using the non-preconditioned method, while computing the quantity $\hat{\sigma}_8^2$ introduced in the present paper takes less than $1$~second and yields a better approximation.

\smallskip

\begin{figure}
\centering
\includegraphics[scale=0.5,trim=0.5cm 2cm 0cm 1cm, clip=true]{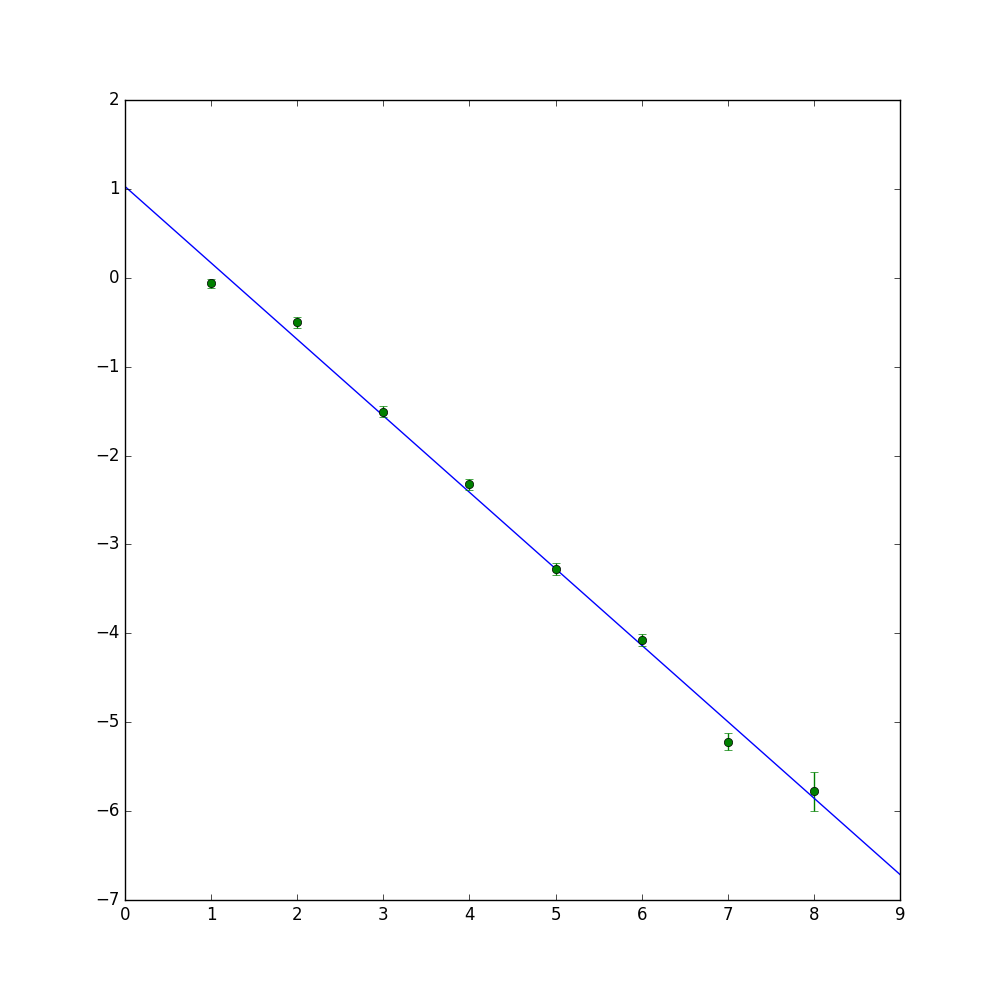}
\caption{
\small{
The base-2 logarithm of the left side of \eqref{e.l2.classical} as a function of $n$. The error bars are $95\%$ confidence intervals. The slope of the regression line is $-0.9$ and the intercept is $1.0$.
}
}
\label{f.error.classical}
\end{figure}

\begin{figure}
\centering
\includegraphics[scale=0.5,trim=0.5cm 2cm 0cm 1cm, clip=true]{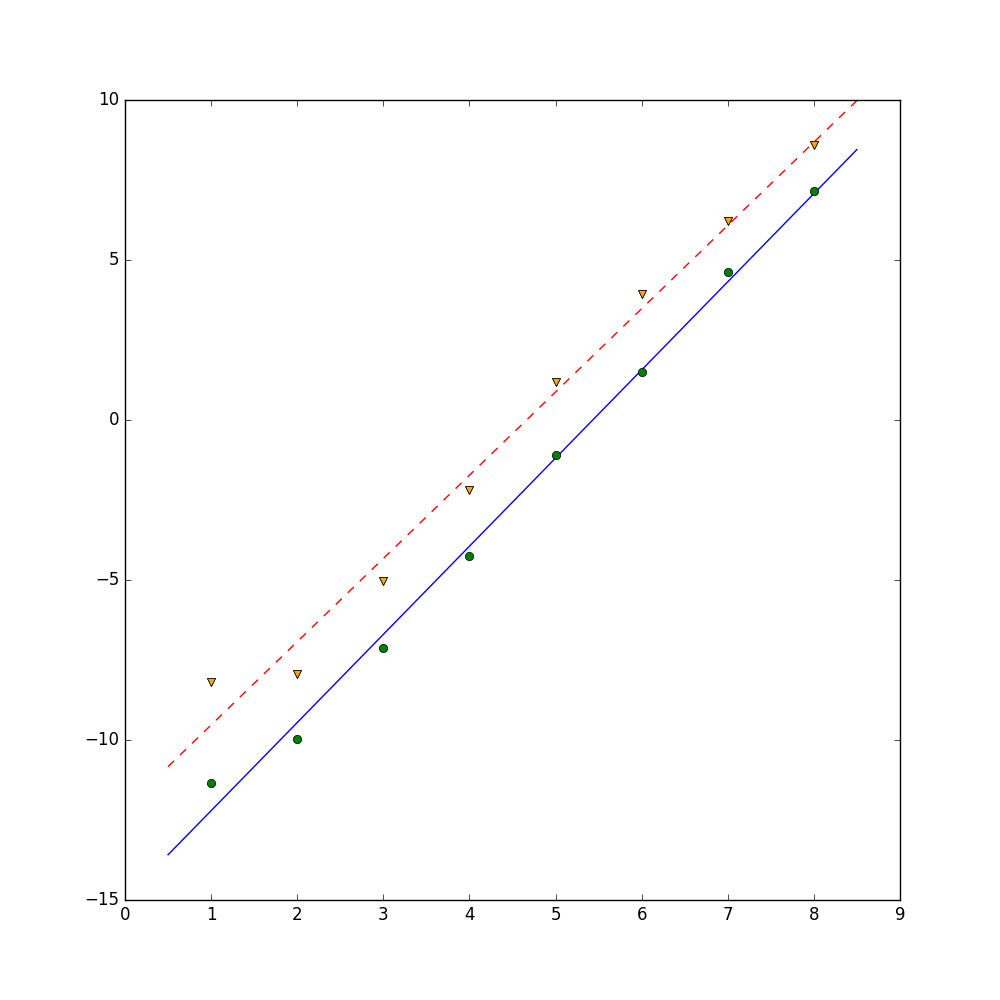}
\caption{
\small{
The base-2 logarithm of the CPU time used to compute $\hat{\sigma}^2_n$, in seconds, as a function of $n$, for the non-preconditioned (green dots and solid regression line) and the preconditioned (orange triangles and dashed regression line) methods. The slope and intercept of the regression line are $2.7$ and $-15.0$ respectively for the non-preconditioned method, and $2.6$ and $-12.1$ for the preconditioned method.
}
}
\label{f.cpu.classical}
\end{figure}


The computational complexity reported here for the older method seems to be consistent with the numerical findings presented in \cite{Gloria}. Indeed, \cite[Figure~14]{Gloria} displays the computational cost of the method as a function of the error, and suggests that the computational time to reach precision $\delta$ scales like $\de^{-2.6}$, in close agreement with the results reported on Figures~\ref{f.error.classical} and~\ref{f.cpu.classical}\footnote{In order to estimate the slope of the regression line of the set of blue dots in \cite[Figure~14]{Gloria}, I first drew the straight line going through the extremal blue dots, and observed that all the other dots are very close to this line. I then measured the coordinates of the extremal dots to be $(3.59, -0.62)$ and $(9.89,-3.06)$ respectively, which yields that the~$dx/dy$ slope of the line is about $-2.58$. Moreover, the seven dots closest to the leftmost blue point are below this line, while the three dots closest to the rightmost blue point are above this line, and thus essentially any other reasonable choice of pair of of points to draw a line and measure the slope from would yield a larger absolute slope.}.

%
%
%
%
%
%

\section{Discrete-time approach}
\label{s.discrete}

In this final section, we explore an alternative method based on a discrete-time dynamics asociated with the operator $\L$. Contrary to the approach presented in Theorem~\ref{t.main}, the method explored here does not seem to generalize to the continuous-space setting. For discrete-space problems, it provides with an approach that does not involve the resolution of linear equations, and which I conjecture to have a computational cost for a precision of $\de > 0$ of the order of $C \de^{-2}$ (with no logarithmic correction) in dimension $d \ge 3$. This would match exactly the complexity lower bound of Proposition~\ref{p.lower.bound}. In dimension $2$, I expect the number of operations to scale like $\de^{-2} \log(\de^{-1})$ to yield a precision of $\de > 0$. 

We set
\begin{equation*}  
\pi(\a) := \sum_{x \sim 0} \a_{0x},
\end{equation*}
and consider the operator 
\begin{equation*}  
\mathcal L_1 :=  -\pi^{-1} D \cdot \a D.
\end{equation*}
This operator is the generator of a discrete-time Markov chain on $\Omega$, and is self-adjoint on $L^2(\Omega, \pi(\a) \, \d \P(\a))$. We denote by $\la \cdot, \cdot \ra_\pi$ the scalar product in this $L^2$ space and set
\begin{equation}  
\label{e.def.g}
g(\a) := \pi^{-1} D \cdot \a \xi.
\end{equation}
\begin{proposition}
\label{p.discrete.stuff}
We have
\begin{equation}  
\label{e.discrete.stuff}
\xi \cdot \ahom \xi = \E[\xi \cdot \a \xi] - \la g, \L_1^{-1} g \ra_\pi.
\end{equation}
\end{proposition}
\begin{proof}
We recall that the rightmost term in \eqref{e.discrete.stuff} is interpreted as
\begin{equation*}  
\lim_{\mu \to 0} \la g, (\mu + \L_1)^{-1} g \ra_\pi,
\end{equation*}
and that 
\begin{equation*}  
 \E[\xi \cdot \a \xi] - \xi \cdot \ahom \xi = \E[D\phi \cdot \a D \phi],
\end{equation*}
where $D \phi$ is the $L^2$ limit of $D \phi_\mu$, see \eqref{e.conv.phimu}. 
Let $\psi_\mu$ solve 
\begin{equation*}  
\Ll( \mu + \L_1 \Rr) \psi_\mu = g,
\end{equation*}
that is,
\begin{equation*}  
\Ll( \mu - \pi^{-1} D \cdot \a D  \Rr) \psi_\mu = \pi^{-1} D \cdot \a \xi.
\end{equation*}
By the same arguments allowing to justify \eqref{e.conv.phimu}, we can verify that
\begin{equation*}  
D \psi_\mu \xrightarrow[\mu \to 0]{L^2(\Omega)} D \phi,
\end{equation*}
and that $\mu^\frac 1 2 \psi_\mu$ converges to $0$ in $L^2(\P)$. (This uses that $\pi$ is uniformly bounded from above and below.) In particular,
\begin{equation*}  
\E \Ll[ D \phi \cdot \a D \phi \Rr] = \lim_{\mu \to 0} \E \Ll[ D \psi_\mu \cdot \a D \psi_\mu \Rr] .
\end{equation*}
Moreover, by the weak formulation of the equation for $\psi_\mu$, we have
\begin{align*}  
\la g, (\mu + \L_1)^{-1} g \ra_\pi & = \la g, \psi_\mu \ra_\pi \\
& = -\E \Ll[ \a \xi \cdot D \psi_\mu \Rr]  \\
& = \E \Ll[ D\psi_\mu \cdot \a D \psi_\mu \Rr]  + \mu \E \Ll[ \pi \psi_\mu^2\Rr] \\
& \xrightarrow[\mu \to 0]{} \E[D \phi \cdot \a D \phi]. \qedhere
\end{align*}
\end{proof}

\begin{figure}
\centering
\includegraphics[scale=0.5,trim=0.5cm 2cm 0cm 1cm, clip=true]{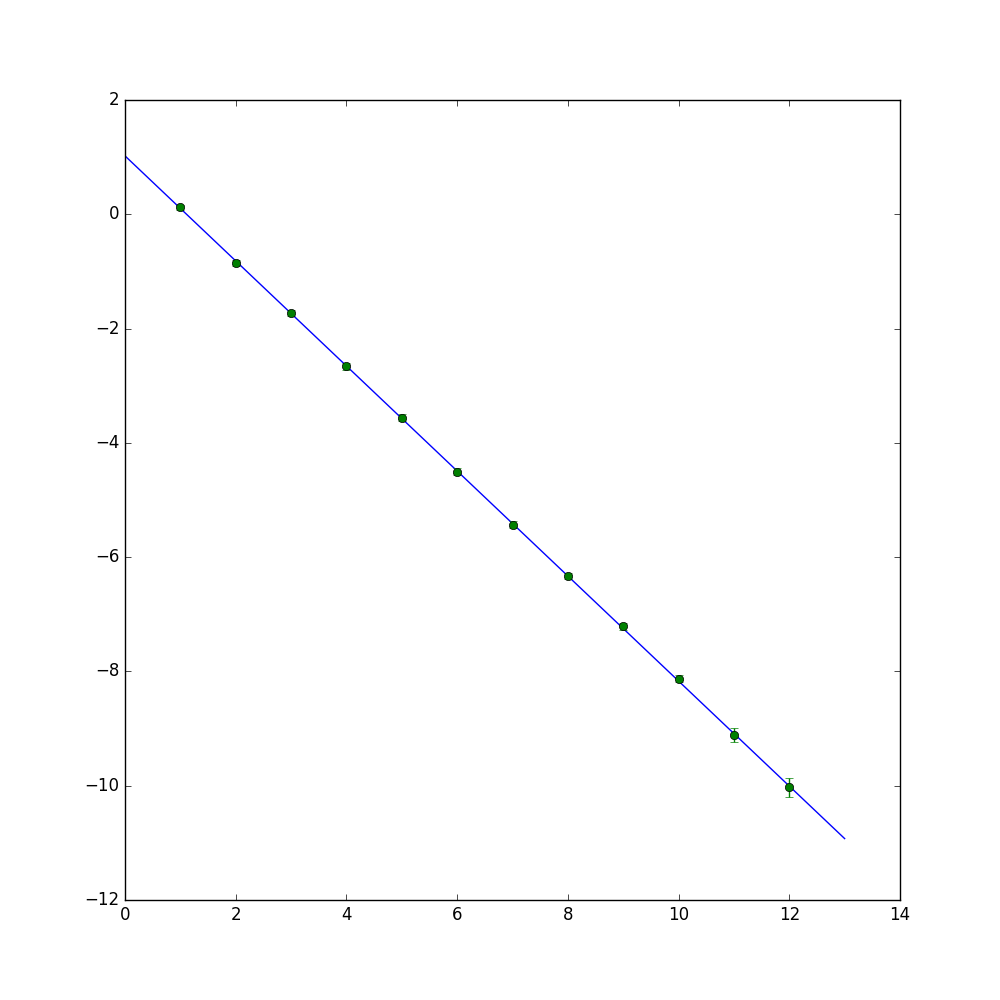}
\caption{
\small{
The base-2 logarithm of the quantity on the left side of \eqref{e.error.discrete}, as a function of $L$. The error bars are $95\%$ confidence intervals. The slope of the regression line is $-0.9$ and the intercept is $1.0$.
}
}
\label{f.error.discrete}
\end{figure}

\begin{figure}
\centering
\includegraphics[scale=0.5,trim=0.5cm 2cm 0cm 1cm, clip=true]{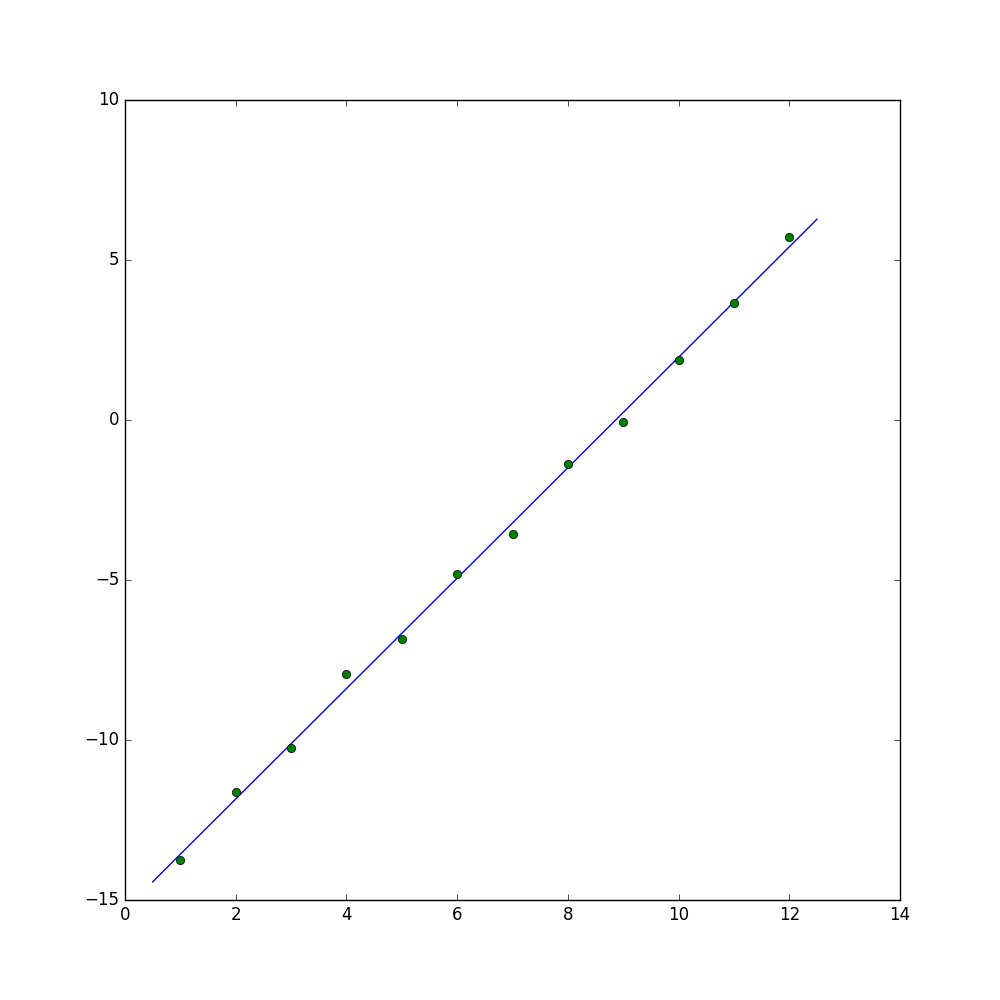}
\caption{
\small{
The base-2 logarithm of the CPU time used to compute $D_L$, in seconds, as a function of $L$. The slope of the regression line is $1.7$ and the intercept is $-15.3$.
}
}
\label{f.cpu.discrete}
\end{figure}

Note that the operator $\mathrm{Id} - \L_1$ is a contraction in $L^\infty(\Omega)$. It is also self-adjoint, and its spectrum is therefore contained in $[-1,1]$. The spectrum of $\L_1$ is thus contained in $[0,2]$. In order to avoid problems related to the periodicity of the underlying Markov chain, we modify slightly the formula~\eqref{e.identity.resolv.discrete} for computing~$\la g, \L_1^{-1} g \ra_{\pi}$. We use instead that, for every~$\lambda \in [0,2]$,
\begin{equation}  
\label{e.another.identity}
\frac 1 {\lambda} = \frac 1 2\sum_{k = 0}^{+\infty}\Ll(1-\frac{\lambda}{2}\Rr)^k 
,
\end{equation}
so that, setting
\begin{equation*}  
\mathcal{P} := \mathrm{Id} - \frac 1 2 \L_1,
\end{equation*}
we have by spectral calculus that
\begin{align}  
\notag
\la g, \L_1^{-1} g \ra_\pi & = \frac 1 2 \sum_{k = 0}^{+\infty} \la g, \mathcal P^k g \ra_\pi \\
\label{e.self.mclT}
& = \frac 1 2 \sum_{k = 0}^{+\infty} \Ll(\la \mathcal P^k g, \mathcal P^k g \ra_\pi + \la \mathcal P^k g, \mathcal P^{k+1} g \ra_\pi\Rr).
\end{align}
In view of the rigorous results in the continuous-time case, we expect that
\begin{equation}  
\label{e.decay.discrete}
\Ll| \la g, \mathcal P^k g \ra_\pi \Rr| \le C (1+k)^{-1 - \frac d 2}.
\end{equation}
If we write $w(k,x) := (\mcl P^k g)(\theta_x \a)$, then the function $w$ solves the discrete-time parabolic equation
\begin{equation}
\label{e.disc.parab}
\Ll\{
\renewcommand{\arraystretch}{1.5}
\begin{array}{rll}  
 w(k+1,x) & =  w(k,x) + \frac 1 {2\pi(x)} \nabla \cdot \a \nabla  w(k,x)  & \text{ for } (k,x) \in \N \times \Zd, \\
 w(0,x) & = \frac 1 {\pi(x)} (\nabla \cdot \a \xi)(x), & \text {for } x \in \Zd,
\end{array}
\Rr.
\end{equation}
where we write $\pi(x) := \pi(\theta_x \a)$. In analogy with Theorem~\ref{t.main} and Remark~\ref{e.sqrt.log}, we define, for each $L \in \N$,
\begin{equation}
\label{e.def.Dl}
D_L := \sum_{\ell = 0}^{L-1} \sum_{k = 2^\ell-1}^{2^{\ell+1}-2} \frac 1 {|B_{r(L,\ell)}|} \sum_{x \in B_{r(L,\ell)}} \pi(x) \Ll( w^2(k,x) + w(k,x) w(k+1,x) \Rr) ,
\end{equation}
where 
\begin{equation*}  
r(L,\ell) := 2^{L - \frac \ell 2} ,
\end{equation*}
and we expect that
\begin{equation}  
\label{e.error.discrete}
\E \Ll[  \Ll( \la g, \L^{-1} g \ra_\pi - D_L  \Rr)^2  \Rr]^\frac 1 2 \le C L^\frac 1 2 2^{-\frac{dL}{2}}.
\end{equation}
Moreover, the extraneous factor of $L^\frac 1 2$ is expected to be absent if we redefine $r(L,\ell)$ to be $2^{L - \frac \ell 2 + \eps \ell}$, for $\eps > 0$ arbitarily small. Finally, the computational cost of this scheme is given by
\begin{align*}  
\sum_{l = 0}^{L-1} \sum_{k = 2^{\ell} - 1}^{2^{\ell + 1} - 2} |B_{r(L,\ell)}| \le C  \sum_{l = 0}^{L-1} \sum_{k = 2^{\ell} - 1}^{2^{\ell + 1} - 2} 2^{d \Ll( L - \frac {\ell}{2} \Rr) } \le C 2^{dL} 
\left|
\begin{array}{ll}
L & \qquad \text{if } d = 2, \\
1 & \qquad \text{if } d \ge 3.
\end{array}
\Rr.
\end{align*}
If we replace $r(L,\ell)$ by $2^{L - \frac \ell 2 + \eps \ell}$, then the complexity of this method degrades in $d = 2$, no matter how small $\eps > 0$ is chosen. This is in contrast with the method explored in the previous sections. On the other hand, in dimension $d \ge 3$ and for $\eps > 0$ sufficiently small, the complexity remains bounded by $C 2^{dL}$ and we that the error bound in \eqref{e.error.discrete} holds with right-hand side replaced by $2^{-\frac {dL}{2}}$. We thus expect this scheme to match exactly the complexity lower bound given by Proposition~\ref{p.lower.bound} in dimensions $d \ge 3$. 

\smallskip

We now report on numerical tests related to the quantity $D_L$ defined in~\eqref{e.def.Dl}. The code can be downloaded as part of the source files of the arXiv posting of this paper. We explore the same two-dimensional example as in Section~\ref{s.discrete}, and use the same computer to run the code. Figure~\ref{f.error.discrete} displays the base-2 logarithm of the left side of \eqref{e.error.discrete}, as a function of $L$, while Figure~\ref{f.cpu.discrete} displays the base-2 logarithm of the required CPU time to compute $D_L$. The numerical findings are very consistent with the heuristic predictions above. It takes less than one minute to compute $D_{12}$. The ratio of the slopes in Figures~\ref{f.cpu.discrete} and \ref{f.error.discrete} is about $1.9$, in close agreement with the predicted theoretical complexity. The fact that we observe an exponent slightly below $2$ may be explained by the diminishing importance taken up by boundary layers as we move to larger values of $L$.



%
%
%
%
%
%

\subsection*{Acknowledgments} I would like to thank Josselin Garnier for an inspiring talk which motivated me to revisit this problem, Tony Leli\`evre for his helpful feedback, and Harmen Stoppels for his precious help with the Julia language. This work has been partially supported by the ANR Grant LSD (ANR-15-CE40-0020-03).

\small
\bibliographystyle{abbrv}
\bibliography{efficient}

\newcommand{\noop}[1]{} \def\cprime{$'$}
\begin{thebibliography}{10}

\bibitem{hmm}
A.~Abdulle, W.~E, B.~Engquist, and E.~Vanden-Eijnden.
\newblock The heterogeneous multiscale method.
\newblock {\em Acta Numer.}, 21:1--87, 2012.

\bibitem{almog1}
Y.~Almog.
\newblock Averaging of dilute random media: a rigorous proof of the
  {C}lausius-{M}ossotti formula.
\newblock {\em Arch. Ration. Mech. Anal.}, 207(3):785--812, 2013.

\bibitem{almog2}
Y.~Almog.
\newblock The {C}lausius-{M}ossotti formula in a dilute random medium with
  fixed volume fraction.
\newblock {\em Multiscale Model. Simul.}, 12(4):1777--1799, 2014.

\bibitem{alb1}
A.~Anantharaman and C.~Le~Bris.
\newblock A numerical approach related to defect-type theories for some weakly
  random problems in homogenization.
\newblock {\em Multiscale Model. Simul.}, 9(2):513--544, 2011.

\bibitem{alb2}
A.~Anantharaman and C.~Le~Bris.
\newblock Elements of mathematical foundations for numerical approaches for
  weakly random homogenization problems.
\newblock {\em Commun. Comput. Phys.}, 11(4):1103--1143, 2012.

\bibitem{time-dep}
D.~Arjmand and O.~Runborg.
\newblock A time dependent approach for removing the cell boundary error in
  elliptic homogenization problems.
\newblock {\em J. Comput. Phys.}, 314:206--227, 2016.

\bibitem{AKMBook}
S.~Armstrong, T.~Kuusi, and J.-C. Mourrat.
\newblock {\em Quantitative stochastic homogenization and large-scale
  regularity}.
\newblock \noop{2018}Preliminary version available at
  {\url{http://perso.ens-lyon.fr/jean-christophe.mourrat/lecturenotes.pdf}}
  (2017).

\bibitem{AKM}
S.~Armstrong, T.~Kuusi, and J.-C. Mourrat.
\newblock Mesoscopic higher regularity and subadditivity in elliptic
  homogenization.
\newblock {\em Comm. Math. Phys.}, 347(2):315--361, 2016.

\bibitem{AKM2}
S.~Armstrong, T.~Kuusi, and J.-C. Mourrat.
\newblock The additive structure of elliptic homogenization.
\newblock {\em Invent. Math.}, 208(3):999--1154, 2017.

\bibitem{AM}
S.~N. Armstrong and J.-C. Mourrat.
\newblock Lipschitz regularity for elliptic equations with random coefficients.
\newblock {\em Arch. Ration. Mech. Anal.}, 219(1):255--348, 2016.

\bibitem{AS}
S.~N. Armstrong and C.~K. Smart.
\newblock Quantitative stochastic homogenization of convex integral
  functionals.
\newblock {\em Ann. Sci. \'Ec. Norm. Sup\'er. (4)}, 49(2):423--481, 2016.

\bibitem{barandco}
M.~T. Barlow, A.~A. J{\'a}rai, T.~Kumagai, and G.~Slade.
\newblock Random walk on the incipient infinite cluster for oriented
  percolation in high dimensions.
\newblock {\em Comm. Math. Phys.}, 278(2):385--431, 2008.

\bibitem{BACF1}
G.~{Ben Arous}, M.~Cabezas, and A.~Fribergh.
\newblock Scaling limit for the ant in high-dimensional labyrinths, \noop{2015}
  {preprint, arXiv:1609.03977}.

\bibitem{BACF2}
G.~{Ben Arous}, M.~Cabezas, and A.~Fribergh.
\newblock Scaling limit for the ant in a simple labyrinth, \noop{2016}
  {preprint, arXiv:1609.03980}.

\bibitem{bm1}
L.~Berlyand and V.~Mityushev.
\newblock Generalized {C}lausius-{M}ossotti formula for random composite with
  circular fibers.
\newblock {\em J. Statist. Phys.}, 102(1-2):115--145, 2001.

\bibitem{blb}
X.~Blanc and C.~Le~Bris.
\newblock Improving on computation of homogenized coefficients in the periodic
  and quasi-periodic settings.
\newblock {\em Netw. Heterog. Media}, 5(1):1--29, 2010.

\bibitem{brandt}
A.~Brandt.
\newblock Multiscale scientific computation: review 2001.
\newblock In {\em Multiscale and multiresolution methods}, volume~20 of {\em
  Lect. Notes Comput. Sci. Eng.}, pages 3--95. Springer, Berlin, 2002.

\bibitem{damronetal}
M.~Damron, J.~Hanson, and P.~Sosoe.
\newblock Subdiffusivity of random walk on the 2{D} invasion percolation
  cluster.
\newblock {\em Stochastic Process. Appl.}, 123(9):3588--3621, 2013.

\bibitem{dolg}
D.~Dolgopyat.
\newblock Limit theorems for partially hyperbolic systems.
\newblock {\em Trans. Amer. Math. Soc.}, 356(4):1637--1689, 2004.

\bibitem{DG}
M.~Duerinckx and A.~Gloria.
\newblock Analyticity of homogenized coefficients under {B}ernoulli
  perturbations and the {C}lausius-{M}ossotti formulas.
\newblock {\em Arch. Ration. Mech. Anal.}, 220(1):297--361, 2016.

\bibitem{eh-book}
Y.~Efendiev and T.~Y. Hou.
\newblock {\em Multiscale finite element methods}, volume~4 of {\em Surveys and
  Tutorials in the Applied Mathematical Sciences}.
\newblock Springer, New York, 2009.

\bibitem{cemracs}
A.-C. Egloffe, A.~Gloria, J.-C. Mourrat, and T.~N. Nguyen.
\newblock Random walk in random environment, corrector equation and homogenized
  coefficients: from theory to numerics, back and forth.
\newblock {\em IMA J. Numer. Anal.}, 35(2):499--545, 2015.

\bibitem{Gloria}
A.~Gloria.
\newblock Numerical approximation of effective coefficients in stochastic
  homogenization of discrete elliptic equations.
\newblock {\em ESAIM Math. Model. Numer. Anal.}, 46(1):1--38, 2012.

\bibitem{GloriaH}
A.~Gloria and Z.~Habibi.
\newblock Reduction in the resonance error in numerical homogenization {II}:
  {C}orrectors and extrapolation.
\newblock {\em Found. Comput. Math.}, 16(1):217--296, 2016.

\bibitem{approx}
A.~Gloria and J.-C. Mourrat.
\newblock Spectral measure and approximation of homogenized coefficients.
\newblock {\em Probab. Theory Related Fields}, 154(1-2):287--326, 2012.

\bibitem{MCdisc}
A.~Gloria and J.-C. Mourrat.
\newblock Quantitative version of the {K}ipnis-{V}aradhan theorem and {M}onte
  {C}arlo approximation of homogenized coefficients.
\newblock {\em Ann. Appl. Probab.}, 23(4):1544--1583, 2013.

\bibitem{GNO}
A.~Gloria, S.~Neukamm, and F.~Otto.
\newblock Quantification of ergodicity in stochastic homogenization: optimal
  bounds via spectral gap on {G}lauber dynamics.
\newblock {\em Invent. Math.}, 199(2):455--515, 2015.

\bibitem{GNO3}
A.~Gloria, S.~Neukamm, and F.~Otto.
\newblock A regularity theory for random elliptic operators, \noop{3002}
  {preprint, arXiv:1409.2678}.

\bibitem{GO1}
A.~Gloria and F.~Otto.
\newblock An optimal variance estimate in stochastic homogenization of discrete
  elliptic equations.
\newblock {\em Ann. Probab.}, 39(3):779--856, 2011.

\bibitem{GO2}
A.~Gloria and F.~Otto.
\newblock An optimal error estimate in stochastic homogenization of discrete
  elliptic equations.
\newblock {\em Ann. Appl. Probab.}, 22(1):1--28, 2012.

\bibitem{GO5}
A.~Gloria and F.~Otto.
\newblock The corrector in stochastic homogenization: optimal rates, stochastic
  integrability, and fluctuations, \noop{2016}{preprint, arXiv:1510.08290}.

\bibitem{henpet}
P.~Henning and D.~Peterseim.
\newblock Oversampling for the multiscale finite element method.
\newblock {\em Multiscale Model. Simul.}, 11(4):1149--1175, 2013.

\bibitem{hw}
T.~Y. Hou and X.-H. Wu.
\newblock A multiscale finite element method for elliptic problems in composite
  materials and porous media.
\newblock {\em J. Comput. Phys.}, 134(1):169--189, 1997.

\bibitem{hwc}
T.~Y. Hou, X.-H. Wu, and Z.~Cai.
\newblock Convergence of a multiscale finite element method for elliptic
  problems with rapidly oscillating coefficients.
\newblock {\em Math. Comp.}, 68(227):913--943, 1999.

\bibitem{hughes}
B.~D. Hughes.
\newblock Conduction and diffusion in percolating systems.
\newblock In {\em Encyclopedia of complexity and systems science}, pages
  1395--1424. Springer, 2009.

\bibitem{jarnac}
A.~A. J{\'a}rai and A.~Nachmias.
\newblock Electrical resistance of the low dimensional critical branching
  random walk.
\newblock {\em Comm. Math. Phys.}, 331(1):67--109, 2014.

\bibitem{kesten-subdiff}
H.~Kesten.
\newblock Subdiffusive behavior of random walk on a random cluster.
\newblock {\em Ann. Inst. H. Poincar\'e Probab. Statist.}, 22(4):425--487,
  1986.

\bibitem{eq-free}
I.~G. Kevrekidis, C.~W. Gear, J.~M. Hyman, P.~G. Kevrekidis, O.~Runborg, and
  C.~Theodoropoulos.
\newblock Equation-free, coarse-grained multiscale computation: enabling
  microscopic simulators to perform system-level analysis.
\newblock {\em Commun. Math. Sci.}, 1(4):715--762, 2003.

\bibitem{kipvar}
C.~Kipnis and S.~R.~S. Varadhan.
\newblock Central limit theorem for additive functionals of reversible {M}arkov
  processes and applications to simple exclusions.
\newblock {\em Comm. Math. Phys.}, 104(1):1--19, 1986.

\bibitem{klo}
T.~Komorowski, C.~Landim, and S.~Olla.
\newblock {\em Fluctuations in {M}arkov processes}, volume 345 of {\em
  Grundlehren der Mathematischen Wissenschaften}.
\newblock Springer, Heidelberg, 2012.

\bibitem{kozlov}
S.~M. Kozlov.
\newblock Geometric aspects of averaging.
\newblock {\em Uspekhi Mat. Nauk}, 44(2(266)):79--120, 1989.

\bibitem{koznac}
G.~Kozma and A.~Nachmias.
\newblock The {A}lexander-{O}rbach conjecture holds in high dimensions.
\newblock {\em Invent. Math.}, 178(3):635--654, 2009.

\bibitem{kum-stflour}
T.~Kumagai.
\newblock {\em Random walks on disordered media and their scaling limits},
  volume 2101 of {\em Lecture Notes in Mathematics}.
\newblock Springer, Cham, 2014.

\bibitem{lebris-survey}
C.~Le~Bris and F.~Legoll.
\newblock Examples of computational approaches for elliptic, possibly
  multiscale {PDE}s with random inputs.
\newblock {\em J. Comput. Phys.}, 328:455--473, 2017.

\bibitem{lig2}
T.~M. Liggett.
\newblock {\em Stochastic interacting systems: contact, voter and exclusion
  processes}, volume 324 of {\em Grundlehren der Mathematischen
  Wissenschaften}.
\newblock Springer-Verlag, Berlin, 1999.

\bibitem{lig1}
T.~M. Liggett.
\newblock {\em Continuous time {M}arkov processes}, volume 113 of {\em Graduate
  Studies in Mathematics}.
\newblock American Mathematical Society, Providence, RI, 2010.

\bibitem{liver}
C.~Liverani.
\newblock Central limit theorem for deterministic systems.
\newblock In {\em International {C}onference on {D}ynamical {S}ystems
  ({M}ontevideo, 1995)}, volume 362 of {\em Pitman Res. Notes Math. Ser.},
  pages 56--75. Longman, Harlow, 1996.

\bibitem{malpet}
A.~M\aa{}lqvist and D.~Peterseim.
\newblock Localization of elliptic multiscale problems.
\newblock {\em Math. Comp.}, 83(290):2583--2603, 2014.

\bibitem{MO}
D.~Marahrens and F.~Otto.
\newblock Annealed estimates on the {G}reen function.
\newblock {\em Probab. Theory Related Fields}, 163(3-4):527--573, 2015.

\bibitem{maxwell}
J.~C. Maxwell.
\newblock Medium in which small spheres are uniformly disseminated.
\newblock \emph{A treatise on electricity and magnetism}, part II, chapter IX,
  article 314. Clarendon Press, 3d ed., 1891.

\bibitem{melbourne}
I.~Melbourne and M.~Nicol.
\newblock Almost sure invariance principle for nonuniformly hyperbolic systems.
\newblock {\em Comm. Math. Phys.}, 260(1):131--146, 2005.

\bibitem{vardecay}
J.-C. Mourrat.
\newblock Variance decay for functionals of the environment viewed by the
  particle.
\newblock {\em Ann. Inst. Henri Poincar\'e Probab. Stat.}, 47(1):294--327,
  2011.

\bibitem{dl-diff}
J.-C. Mourrat.
\newblock First-order expansion of homogenized coefficients under {B}ernoulli
  perturbations.
\newblock {\em J. Math. Pures Appl. (9)}, 103(1):68--101, 2015.

\bibitem{papvar}
G.~Papanicolaou and S.~R.~S. Varadhan.
\newblock Ornstein-{U}hlenbeck process in a random potential.
\newblock {\em Comm. Pure Appl. Math.}, 38(6):819--834, 1985.

\bibitem{papa-survey}
G.~C. Papanicolaou.
\newblock Diffusion in random media.
\newblock In {\em Surveys in applied mathematics, {V}ol.~1}, pages 205--253.
  Plenum, New York, 1995.

\bibitem{petrov}
V.~V. Petrov.
\newblock {\em Limit theorems of probability theory}, volume~4 of {\em Oxford
  Studies in Probability}.
\newblock The Clarendon Press, Oxford University Press, New York, 1995.

\bibitem{quart}
A.~Quarteroni, R.~Sacco, and F.~Saleri.
\newblock {\em Numerical mathematics}, volume~37 of {\em Texts in Applied
  Mathematics}.
\newblock Springer-Verlag, New York, 2000.

\bibitem{rayleigh}
J.~W. {Strutt, 3d Baron Rayleigh}.
\newblock On the influence of obstacles arranged in rectangular order upon the
  properties of a medium.
\newblock {\em Philos. mag.}, 34(211):481--502, 1892.

\bibitem{yue-e}
X.~Yue and W.~E.
\newblock The local microscale problem in the multiscale modeling of strongly
  heterogeneous media: effects of boundary conditions and cell size.
\newblock {\em J. Comput. Phys.}, 222(2):556--572, 2007.

\end{thebibliography}

\end{document}